\documentclass{amsproc}

\usepackage{amssymb}

\usepackage{graphicx}

\usepackage{verbatim}

\newtheorem{theorem}{Theorem}[section]
\newtheorem{corollary}[theorem]{Corollary}
\newtheorem{lemma}[theorem]{Lemma}
\newtheorem{proposition}[theorem]{Proposition}

\theoremstyle{definition}

\theoremstyle{remark}

\theoremstyle{definition}
\newtheorem{assumption}[theorem]{Assumption}

\numberwithin{equation}{section}

\def\vu{\textit{\textbf{u}}}
\def\vv{\textit{\textbf{v}}}
\def\vw{\textit{\textbf{w}}}
\def\vf{\textit{\textbf{f}}}
\def\vF{\textit{\textbf{F}}}
\def\vg{\textit{\textbf{g}}}
\def\vh{\textit{\textbf{h}}}

\def\bA{\mathbb{A}}
\def\bO{\mathbb{O}}
\def\bR{\mathbb{R}}
\def\bZ{\mathbb{Z}}

\def\bC{\mathbb{C}}

\def\cA{\mathcal{A}}
\def\cB{\mathcal{B}}

\def\cD{\mathcal{D}}
\def\cE{\mathcal{E}}

\def\cM{\mathcal{M}}

\def\cQ{\mathcal{Q}}

\def\cT{\mathcal{T}}
\def\cL{\mathcal{L}}

\newcommand{\Div}{\operatorname{div}}
\newcommand{\dist}{\text{dist}}

\makeatletter
\def\dashint{\operatorname%
{\,\,\text{\bf--}\kern-.98em\DOTSI\intop\ilimits@\!\!}}
\makeatother

\begin{document}

\title[Higher order elliptic systems]
{The conormal derivative problem for higher order elliptic systems
with irregular coefficients}


\author[H. Dong]{Hongjie Dong}
\address[H. Dong]{Division of Applied Mathematics, Brown University,
182 George Street, Providence, RI 02912, USA}
\email{Hongjie\_Dong@brown.edu}
\thanks{H. Dong was partially supported by the NSF under agreement DMS-0800129 and DMS-1056737.}

\author[D. Kim]{Doyoon Kim}
\address[D. Kim]{Department of Applied Mathematics, Kyung Hee University, 1732 Deogyeong-daero, Giheung-gu, Yongin-si, Gyeonggi-do 446-701, Republic of Korea}
\email{doyoonkim@khu.ac.kr}
\thanks{D. Kim was supported by Basic Science Research Program through the National Research Foundation of Korea (NRF) funded by the Ministry of Education, Science and Technology (2011-0013960).}


\date{}

\begin{abstract}
We prove $L_p$ estimates of solutions to a conormal derivative problem for divergence form complex-valued higher-order elliptic systems on a half space and
on a Reifenberg flat domain. The leading coefficients are assumed
to be merely measurable in one direction and have small mean
oscillations in the orthogonal directions on each small ball. Our results are new even in the second-order case. The corresponding results for the Dirichlet problem were obtained recently in \cite{DK10}.
\end{abstract}

\maketitle

\section{Introduction}

This paper is concerned with $L_p$ theory for higher-order elliptic systems in divergence form with {\em conormal} derivative boundary conditions. Our focus is to seek minimal regularity assumptions on the leading coefficients of elliptic systems defined on regular and irregular domains. The paper is a continuation of \cite{DK09_01,DK10}, where the authors considered higher-order systems in the whole space and on domains with {\em Dirichlet} boundary conditions.

There is a vast literature on $L_p$ theory for second-order and higher-order elliptic and parabolic equations/systems with {\em constant} or {\em uniformly continuous} coefficients. We refer the reader to the classical work \cite{A65,ADN64,Solo,LSU,Fried}. Concerning possibly discontinuous coefficients, a notable class is the set of bounded functions with vanishing mean oscillations (VMO). This class of coefficients was firstly introduced in \cite{CFL1,CFL2} in the case of second-order non-divergence form elliptic equations, and further considered by a number of authors in various contexts, including higher-order equations and systems; see, for instance, \cite{CFF,HHH,Mi06,PS3}.

Recently, in \cite{DK09_01,DK10} the authors studied the {\em Dirichlet} problem for higher-order elliptic and parabolic systems with possibly measurable coefficients. In \cite{DK09_01}, we established the $L_p$-solvability of both divergence and non-divergence form systems with coefficients (called $\text{VMO}_x$ coefficients in \cite{Krylov_2005}) having locally small mean oscillations with respect to the spatial variables, and measurable in the time variable in the parabolic case. While in \cite{DK10}, divergence form elliptic and parabolic systems of arbitrary order are considered in the whole space,  on a half space, and on Reifenberg flat domains, with {\em variably partially BMO coefficients}. This class of coefficients was introduced in \cite{Krylov08} in the context of second-order non-divergence form elliptic equations in the whole space, and naturally appears in the homogenization of layered materials; see, for instance, \cite{CKV}.  It was later considered by the authors of the present article in \cite{DK09_02,DK10} and by Byun and Wang in \cite{BW10}. Loosely speaking, on
each cylinder (or ball in the elliptic case), the coefficients are
allowed to be merely
measurable in one spatial direction called the {\em measurable direction}, which may vary for different cylinders. It is also assumed that the coefficients have small mean oscillations in the orthogonal directions, and near the boundary the measurable direction is sufficiently close to the ``normal'' direction of the boundary. Note that the boundary of a Reifenberg flat domain is locally trapped in thin discs, which allows the boundary to have a fractal structure; cf. \eqref{eq3.49}. Thus the normal direction of the boundary may not be well defined for Reifenberg flat domains, so instead we take the normal direction of the top surface of these thin discs.

The proofs in \cite{DK09_01,DK10} are in the spirit of \cite{Krylov_2005} by N. V. Krylov, in which the author gave a unified approach of $L_p$ estimates for both divergence and non-divergence second-order elliptic and parabolic equations in the whole space with $\text{VMO}_x$ coefficients. One of the crucial steps in \cite{Krylov_2005} is to establish certain interior {\em mean oscillation estimates}\footnote{Also see relevant early work \cite{Iw83,DM93}.} of solutions to equations with ``simple'' coefficients, which are measurable functions of the time variable only. Then the estimates for equations with $\text{VMO}_x$ coefficients follow from the mean oscillation estimates combined with a perturbation argument. In this connection, we point out that in \cite{DK09_01,DK10} a great deal of efforts were made to derive boundary and interior mean oscillation estimates for solutions to higher-order systems. For systems in Reifenberg flat domains, we also used an idea in \cite{CaPe98}.

In this paper, we study a {\em conormal} derivative problem for elliptic operators in
divergence form of order $2m$:
\begin{equation}                             \label{eq0617_02}
\cL \vu :=\sum_{|\alpha|\le m,|\beta|\le
m}D^\alpha(a_{\alpha\beta}D^\beta\vu),
\end{equation}
where $\alpha$ and $\beta$ are $d$-dimensional multi-indices,
$a_{\alpha\beta}=[a_{\alpha\beta}^{ij}(x)]_{i,j=1}^n$ are $n\times n$ complex matrix-valued functions,
and $\vu$ is a complex vector-valued function.
For $\alpha=(\alpha_1,\ldots,\alpha_d)$, we use the notation
$D^\alpha \vu =D_1^{\alpha_1}\ldots D_d^{\alpha_d} \vu$.
All the coefficients are assumed to be bounded and measurable, and
$\cL$ is uniformly elliptic; cf. \eqref{eq11.28}. Consider the following elliptic system
\begin{equation}
                                        \label{eq9.15}
(-1)^m\cL \vu+\lambda \vu=\sum_{|\alpha|\le m}D^\alpha\vf_\alpha
\end{equation}
on a domain $\Omega$ in $\bR^d$, where $\vf_\alpha \in L_p(\Omega)$, $p\in (1,\infty)$, and $\lambda\ge 0$ is a constant. A function $\vu\in W^m_p$ is said to be a weak solution to \eqref{eq9.15} on $\Omega$ with the conormal derivative boundary condition associated with $\vf_\alpha$ (on $\partial \Omega$) if
\begin{equation}
                                \label{eq3.02}
\int_{\Omega}\sum_{|\alpha|\le m,|\beta|\le m}(-1)^{m+|\alpha|} D^\alpha\phi \cdot
a_{\alpha\beta}D^\beta\vu +\lambda \phi\cdot \vu\,dx
=\sum_{|\alpha|\le m}\int_{\Omega}(-1)^{|\alpha|}D^\alpha\phi\cdot
\vf_\alpha\,dx
\end{equation}
for any test function $\phi=(\phi^1,\phi^2,\ldots,\phi^n)\in W^m_q(\Omega)$, where $q=p/(p-1)$. We emphasize that the phrase ``associated with $\vf_\alpha$'' is appended after ``the conormal derivative boundary condition'' because for different representations of the right-hand side of \eqref{eq9.15}, even if they are pointwise equal, the weak formulation \eqref{eq3.02} could still be different.
In the sequel, we omit this phrase when there is no confusion. We note that the equation above can also be understood as
$$
\int_{\Omega}\sum_{|\alpha|\le m,|\beta|\le m}(-1)^{m+|\alpha|} D^\alpha\phi \cdot
a_{\alpha\beta}D^\beta\vu +\lambda \phi\cdot \vu\,dx=\vF(\phi)\quad \forall \phi\in W^m_{q}(\Omega),
$$
where $\vF$ is a given vector-valued bounded linear functional on $W^m_{q}(\Omega)$. The main objective of the paper is to show the unique $W^m_p(\Omega)$-solvability of \eqref{eq9.15} on a half space or on a possibly unbounded Reifenberg domain with the same regularity conditions on the leading coefficients, that is, {\em variably partially BMO coefficients}, as those in \cite{DK10}. See Section \ref{sec082001} for the precise statements of the assumptions and main results.

Notably, our results are new even for second-order scalar equations. In the literature, an $L_p$ estimate for the conormal derivative problem can be found in \cite{BW05}, where the authors consider second-order divergence elliptic equations without lower-order terms and with coefficients small BMO with respect to all variables on bounded Reifenberg domains. The proof in \cite{BW05} contains a compactness argument, which does not apply to equations with coefficients measurable in some direction discussed in the current paper. For other results about the conormal derivative problem, we refer the reader to \cite{Lieb1} and \cite{Lieb2}.

We prove the main theorems by following the strategy in \cite{DK10}. First, for systems with homogeneous right-hand side and coefficients measurable in one direction, we estimate the H\"older norm of certain linear combinations of $D^m\vu$ in the interior of the domain, as well as near the boundary if the boundary is flat and perpendicular to the measurable direction.  Then by using the H\"{o}lder estimates, we proceed to establish mean oscillation estimates of solutions to elliptic systems. As is expected, the obstruction is in the boundary mean oscillation estimates, to which we give a more detailed account. Note that when obtaining mean oscillation estimates of solutions, even in the half space case we do not require the measurable direction to be exactly perpendicular to the boundary, but allow it to be sufficiently close to the normal direction. For the Dirichlet problem in \cite{DK10}, we used a delicate cut-off argument together with a generalized Hardy's inequality. However, this method no longer works for the conormal derivative problem as solutions do not vanish on the boundary. The key observation in this paper is Lemma \ref{lem4.2} which shows that if one modifies the right-hand side a little bit, then the function $\vu$ itself still satisfies the system with the conormal derivative boundary condition on a subdomain with a flat boundary perpendicular to the measurable direction. This argument is also readily adapted to elliptic systems on Reifenberg flat domains with variably partially BMO coefficients.

The corresponding parabolic problem, however, seems to be still out of reach by the argument mentioned above. In fact, in the modified equation in Lemma \ref{lem4.2} there would be an extra term involving $\vu_t$ on the right-hand side. At the time of this writing, it is not clear to us how to estimate this term.

The remaining part of the paper is organized as follows. We state the main theorems in the next section. Section \ref{sec_aux} contains some auxiliary results including $L_2$-estimates, interior and boundary H\"older estimates, and approximations of Reifenberg domains. In Section \ref{sec4} we establish the interior and boundary mean oscillation estimates and then prove the solvability of systems on a half space. Finally we deal with elliptic systems on a Reifenberg flat domain in Section \ref{Reifenberg}.

We finish the introduction by fixing
some notation. By $\bR^d$ we mean a $d$-dimensional
Euclidean space, a point in $\bR^d$ is denoted by
$x=(x_1,\ldots,x_d)=(x_1,x')$, and $\{e_j\}_{j=1}^d$ is the
standard basis of $\bR^d$. Throughout the paper, $\Omega$
indicates an open set in $\bR^d$. For vectors $\xi,\eta\in \bC^n$,
we denote
$$
(\xi,\eta)=\sum_{i=1}^n \xi^i\overline{\eta^i}.
$$
For a function $f$ defined on a subset $\cD$ in $\bR^{d}$, we set
\begin{equation*}
(f)_{\cD} = \dashint_{\cD}
f(x) \, dx= \frac{1}{|\cD|} \int_{\cD} f(x) \, dx,
\end{equation*}
where $|\cD|$ is the $d$-dimensional Lebesgue measure of $\cD$.  Denote
\begin{align*}
\bR^d_+ &= \{(x_1,x') \in \bR^d: x_1 > 0\},\\
B_r(x) &= \{ y \in \bR^d: |x-y| < r\},\quad
B'_r(x') = \{ y' \in
\bR^{d-1}: |x'-y'| < r\},\\
B_r^+(x)&=B_r(x)\cap \bR^d_+,\quad
\Gamma_r(x)=B_r(x)\cap \partial\bR^d_+,\quad \Omega_r(x)=B_r(x)\cap \Omega.
\end{align*}
For a domain $\Omega$ in $\bR^d$, we define the solution
spaces $W_p^m(\Omega)$ as follows:
\begin{equation*}
W_p^m(\Omega) =\{u\in L_p(\Omega): D^{\alpha}u \in L_p(\Omega), 1
\le |\alpha| \le m \},
\end{equation*}
$$
\|u\|_{W_p^m(\Omega)} = \sum_{|\alpha|\le m}
\|D^{\alpha}u\|_{L_p(\Omega)}.
$$
We denote $C_{\text{loc}}^{\infty}(\cD)$ to be the set of all infinitely differentiable functions on $\cD$, and $C_0^{\infty}(\cD)$ the
set of infinitely differentiable functions with compact
support $\Subset \cD$.

\section{Main results} \label{sec082001}

Throughout the paper, we assume that the $n \times n$
complex-valued coefficient matrices $a_{\alpha\beta}$ are
measurable and bounded, and the leading coefficients
$a_{\alpha\beta}$, $|\alpha|=|\beta|=m$, satisfy an ellipticity
condition. More precisely, we assume:
\begin{enumerate}
\item There exists a constant $\delta \in (0,1)$ such that the
leading coefficients $a_{\alpha\beta}$, $|\alpha|=|\beta|=m$, satisfy
\begin{equation}
                            \label{eq11.28}
\delta |\xi|^2 \le
\sum_{|\alpha|=|\beta|=m}\Re(a_{\alpha\beta}(x) \xi_{\beta},
\xi_{\alpha}),
\quad
|a_{\alpha\beta}| \le \delta^{-1}
\end{equation}
for any $ x \in \bR^{d}$ and $\xi =
(\xi_{\alpha})_{|\alpha|=m}$, $\xi_{\alpha} \in \bC^n$.
Here we use $\Re(f)$ to denote the real part of $f$.

\item All the lower-order coefficients $a_{\alpha\beta}$,
$|\alpha| \ne m$ or $|\beta| \ne m$, are bounded by a constant
$K\ge 1$.
\end{enumerate}

We note that the ellipticity condition \eqref{eq11.28} can be relaxed. For instance, the operator $\cL=D_1^4+D_2^4$ is allowed when $d=m=2$. See Remark 2.5 of \cite{DK10}.

Throughout the paper we
write $\{\bar{a}_{\alpha\beta}\}_{|\alpha|=|\beta|=m} \in \bA$
whenever the $n \times n$ complex-valued matrices
$\bar{a}_{\alpha\beta}=\bar{a}_{\alpha\beta}(y_1)$ are measurable functions satisfying the condition \eqref{eq11.28}. For a linear map $\cT$ from
$\bR^d$ to $\bR^d$, we write $\cT \in \bO$ if $\cT$ is of the form
$$
\cT(x) = \rho x + \xi,
$$
where $\rho$ is a $d \times d$ orthogonal matrix and $\xi \in
\bR^d$.

Let $\cL$ be the elliptic operator defined in \eqref{eq0617_02}. Our first result is about the conormal derivative problem on a
half space. The following mild regularity assumption is imposed on the leading coefficients, with a parameter $\gamma \in (0,1/4)$ to be determined later.
\begin{assumption}[$\gamma$]                          \label{assumption20100901}
There is a constant $R_0\in (0,1]$ such that the following hold
with  $B:=B_r(x_0)$.

(i) For any $x_0\in \bR^d_+$ and any $r\in
\left(0,R_0\right]$ so that
$B\subset \bR^{d}_+$, one can find $\cT_B \in \bO$ and coefficient
matrices $\{\bar a_{\alpha\beta}\}_{|\alpha|=|\beta|=m} \in \bA$
satisfying
\begin{equation}
                                \label{eq10_23}
\sup_{|\alpha|=|\beta|=m}\int_B |a_{\alpha\beta}(x) -
\bar{a}_{\alpha\beta}(y_1)| \, dx \le \gamma |B|,
\end{equation}
where $y = \cT_B (x)$.

(ii) For any $x_0\in \partial \bR^d_+$ and any $r\in (0,R_0]$, one
can find $\cT_B \in \bO$ satisfying $\rho_{11}\ge \cos (\gamma/2)$
and coefficient matrices $\{\bar
a_{\alpha\beta}\}_{|\alpha|=|\beta|=m} \in \bA$ satisfying
\eqref{eq10_23}.
\end{assumption}

The condition $\rho_{11}\ge \cos (\gamma/2)$ with a sufficiently small $\gamma$ means that at any boundary point the $y_1$-direction is sufficiently close to the $x_1$-direction, i.e., the normal direction of the boundary.

\begin{theorem}[Systems on a half space]
                                    \label{thm3}
Let $\Omega=\bR^d_+$, $p \in (1,\infty)$, and
$$
\vf_\alpha= (f_\alpha^1, \ldots, f_\alpha^n)^{\text{tr}} \in
L_p(\Omega), \quad |\alpha|\le m.
$$
Then there exists a constant $\gamma=\gamma(d,n,m,p,\delta)$
such that, under Assumption \ref{assumption20100901} ($\gamma$),
the following hold true.

\noindent (i) For any $\vu \in W^m_p(\Omega)$ satisfying
\begin{equation}
                                    \label{eq1.55}
(-1)^m\cL \vu +\lambda \vu = \sum_{|\alpha|\le m}D^\alpha
\vf_\alpha
\end{equation} in $\Omega$ and the conormal derivative condition on $\partial\Omega$,
we have
\begin{equation*}
\sum_{|\alpha|\le m}\lambda^{1-\frac {|\alpha|} {2m}} \|D^\alpha
\vu \|_{L_p(\Omega)} \le N \sum_{|\alpha|\le m}\lambda^{\frac
{|\alpha|} {2m}} \| \vf_\alpha \|_{L_p(\Omega)},
\end{equation*}
provided that $\lambda \ge \lambda_0$,
where $N$ and $\lambda_0 \ge 0$
depend only on $d$, $n$, $m$, $p$, $\delta$, $K$ and $R_0$.

\noindent (ii) For any  $\lambda > \lambda_0$, there exists a
unique solution $\vu \in W_p^m(\Omega)$ to \eqref{eq1.55} with the
conormal derivative boundary condition.

\noindent
(iii)
If all the lower-order coefficients of $\cL$ are zero and the leading coefficients are measurable functions of $x_1\in \bR$ only, then one can take $\lambda_0=0$.
\end{theorem}

For elliptic systems on a Reifenberg flat domain which is possibly unbounded,  we impose a similar regularity
assumption on $a_{\alpha\beta}$ as in Assumption
\ref{assumption20100901}. Near the boundary, we require that in
each small scale the direction in which the coefficients are only
measurable coincides with the ``normal'' direction of a certain
thin disc, which contains a portion of $\partial\Omega$. More
precisely, we assume the following, where the parameter $\gamma\in
(0,1/50)$ will be determined later.
\begin{assumption}[$\gamma$]
                                        \label{assump1}
There is a constant $R_0\in (0,1]$ such that the following hold.

(i) For any $x\in \Omega$ and any $r\in
(0,R_0]$ such that $B_r(x)\subset
\Omega$, there is an orthogonal coordinate system depending on $x$ and $r$
such that in this new coordinate system, we have
\begin{equation}
                            \label{eq13.07}
\dashint_{B_r(x)}\Big| a_{\alpha\beta}(y_1, y') -
\dashint_{B'_r(x')} a_{\alpha\beta}(y_1,z') \, dz' \Big| \, dy\le
\gamma.
\end{equation}

(ii) The domain $\Omega$ is Reifenberg flat: for any $x\in \partial\Omega$ and $r\in
(0,R_0]$, there is an orthogonal coordinate system depending on $x$ and $r$
such that in this new coordinate system, we have \eqref{eq13.07}
and
\begin{equation}
                    \label{eq3.49}
 \{(y_1,y'):x_1+\gamma r<y_1\}\cap B_r(x)
 \subset\Omega_r(x)
 \subset \{(y_1,y'):x_1-\gamma r<y_1\}\cap B_r(x).
\end{equation}
\end{assumption}

In particular, if the boundary $\partial
\Omega$ is locally the graph of a Lipschitz continuous function
with a small Lipschitz constant, then $\Omega$ is Reifenberg flat.
Thus all $C^1$ domains are Reifenberg flat for any $\gamma>0$.

The next theorem is about the conormal derivative problem on a
Reifenberg flat domain.

\begin{theorem}[Systems on a Reifenberg flat domain]
                                \label{thm5}
Let $\Omega$ be a domain in $\bR^d$ and $p \in (1,\infty)$.
Then there exists a constant
$\gamma=\gamma(d,n,m,p,\delta)$ such that, under Assumption
\ref{assump1} ($\gamma$), the following hold true.

\noindent (i) Let $\vf_\alpha= (f_\alpha^1, \ldots, f_\alpha^n)^{\text{tr}} \in
L_p(\Omega)$, $|\alpha|\le m$. For any $\vu \in W^m_p(\Omega)$ satisfying
\begin{equation}
                                    \label{eq11_01}
(-1)^m \cL \vu +\lambda \vu = \sum_{|\alpha|\le m}D^\alpha
\vf_\alpha \quad \text{in}\quad \Omega
\end{equation}
with the conormal derivative condition on $\partial\Omega$, we
have
\begin{equation*}
\sum_{|\alpha|\le m}\lambda^{1-\frac {|\alpha|} {2m}} \|D^\alpha
\vu \|_{L_p(\Omega)} \le N \sum_{|\alpha|\le m}\lambda^{\frac
{|\alpha|} {2m}} \| \vf_\alpha \|_{L_p(\Omega)},
\end{equation*}
provided that $\lambda \ge \lambda_0$, where $N$ and $\lambda_0
\ge 0$ depend only on $d$, $n$, $m$, $p$, $\delta$, $K$, and
$R_0$.

\noindent (ii) For any  $\lambda > \lambda_0$ and $\vf_\alpha \in
L_p(\Omega)$, $|\alpha|\le m$, there exists a
unique solution $\vu \in W_p^m(\Omega)$ to \eqref{eq11_01} with
the conormal derivative boundary condition.
\end{theorem}

For $\lambda=0$, we have the following solvability result for systems without lower-order terms on bounded domains.

\begin{corollary}
                                \label{cor7}
Let $\Omega$ be a bounded domain in $\bR^d$, and $p \in (1,\infty)$.
Assume that $a_{\alpha\beta}\equiv 0$ for any $\alpha,\beta$ satisfying $|\alpha|+|\beta|<2m$.
Then there exists a constant
$\gamma=\gamma(d,n,m,p,\delta)$ such that, under Assumption
\ref{assump1} ($\gamma$), for any $\vf_\alpha \in
L_p(\Omega)$, $|\alpha|= m$, there exists a solution $\vu \in
W_p^m(\Omega)$ to
\begin{equation}
                                    \label{eq22.34}
(-1)^m \cL \vu = \sum_{|\alpha|= m}D^\alpha
\vf_\alpha \quad \text{in}\quad \Omega
\end{equation}
with the conormal derivative
boundary condition, and $\vu$ satisfies
\begin{equation}
                                \label{eq23.08}
\|D^m \vu \|_{L_p(\Omega)} \le N \sum_{|\alpha|= m}\|\vf_\alpha
\|_{L_p(\Omega)},
\end{equation}
where $N$ depends only on $d$, $n$, $m$, $p$, $\delta$, $K$,
$R_0$, and $|\Omega|$.
Such a solution is unique up to a polynomial of order at most $m-1$.
\end{corollary}

Finally, we present a result for second-order scalar elliptic equations in the form
\begin{equation}
                            \label{eq1005_2}
D_i(a_{ij}D_j u)+D_i(a_i u)+b_iD_i u+cu=\Div g+f\quad \text{in}\,\,\Omega
\end{equation}
with the conormal derivative boundary condition.
The result generalizes Theorem 5 of \cite{DongKim08a}, in which bounded Lipschitz domains with small Lipschitz constants are considered. It also extends the main result of \cite{BW05} to equations with lower-order terms and with leading coefficients in a more general class.
In the theorem below we assume that all the coefficients are bounded and measurable, and $a_{ij}$ satisfies \eqref{eq11.28} with $m=1$. As usual, we say that $D_i a_i+c\le 0$ in $\Omega$ holds in the weak sense if
$$
\int_\Omega(-a_i D_i\phi+c\phi)\,dx\le 0
$$
for any nonnegative $\phi\in C_0^\infty(\Omega)$. By Assumption ($\text{H}$) we mean that
$$
\int_\Omega(-a_i D_i\phi+c\phi)\,dx=0\quad \forall \phi\in C^\infty(\overline{\Omega}).
$$
Similarly, Assumption ($\text{H}^*$) is satisfied if
$$
\int_\Omega(b_i D_i\phi+c\phi)\,dx=0\quad \forall \phi\in C^\infty(\overline{\Omega}).
$$

\begin{theorem}[Scalar equations on a bounded domain]
                                        \label{thmB}
Let $p\in (1,\infty)$ and $\Omega$ be a bounded domain. Assume $D_ia_i+c\le 0$ in $\Omega$ in the weak sense. Then there exists a constant $\gamma=\gamma(d,p,\delta)$  such that, under Assumption \ref{assump1} ($\gamma$), the following hold true.

\noindent
(i) If Assumption ($\text{H}$) is satisfied, then for any $f$, $g = (g_1, \cdots, g_d) \in L_{p}(\Omega)$, the equation \eqref{eq1005_2} has a unique up to a constant solution $u\in W^1_p(\Omega)$ provided that Assumption ($\text{H}^*$) is also satisfied.
Moreover, we have
\begin{equation*}
\|Du\|_{L_p(\Omega)}\le N\|f\|_{L_p(\Omega)}+N\|g\|_{L_p(\Omega)}.
\end{equation*}

\noindent
(ii) If Assumption ($\text{H}$) is not satisfied, the solution is unique and we have
\begin{equation*}
\|u\|_{W^1_p(\Omega)}\le N\|f\|_{L_p(\Omega)}+N\|g\|_{L_p(\Omega)}.
\end{equation*}
The constants $N$ are independent of $f$, $g$, and $u$.
\end{theorem}

\section{Some auxiliary estimates} \label{sec_aux}

In this section we consider operators without lower-order terms. Denote
$$
\cL_0 \vu = \sum_{|\alpha|=|\beta|=m}D^\alpha( a_{\alpha\beta} D^\beta \vu).
$$

\subsection{$L_2$-estimates}                    \label{sec3.1}
The following $L_2$-estimate for elliptic operators in divergence form with measurable coefficients is classical. We give a sketched proof for the sake of completeness.

\begin{theorem}         \label{theorem08061901}
Let $\Omega = \bR^d$ or $\bR^d_+$.
There exists $N = N(d,m,n, \delta)$
such that, for any $\lambda \ge 0$,
\begin{equation}
                                \label{eq2010_01}
\sum_{|\alpha|\le m}\lambda^{1-\frac {|\alpha|} {2m}} \|D^\alpha \vu \|_{L_2(\Omega)}
\le N \sum_{|\alpha|\le m}\lambda^{\frac {|\alpha|} {2m}} \| \vf_\alpha \|_{L_2(\Omega)},
\end{equation}
provided that $\vu \in W_2^m(\Omega)$ and $\vf_\alpha \in L_2(\Omega)$, $|\alpha|\le m$, satisfy
\begin{equation}                             \label{eq080501}
(-1)^m\cL_0 \vu + \lambda \vu = \sum_{|\alpha|\le m}D^\alpha \vf_\alpha
\end{equation}
in $\Omega$ with the conormal derivative condition on $\partial\Omega$.
Furthermore, for any $\lambda > 0$ and $\vf_\alpha \in L_2(\Omega),|\alpha|\le m$, there exists a unique solution $\vu\in W_2^m(\Omega)$ to
the equation \eqref{eq080501} in $\Omega$ with the conormal derivative boundary condition.
\end{theorem}
\begin{proof}
By the method of continuity and a standard density argument, it suffices to prove the estimate \eqref{eq2010_01} for $\vu \in C^{\infty}(\overline{\Omega})\cap W_2^m(\Omega)$.
From the equation, it follows that
\begin{equation*}
\int_{\Omega} \left[ (D^{\alpha}\vu, a_{\alpha\beta}D^{\beta}\vu)
+ \lambda |\vu|^2 \right]\,dx
= \sum_{|\alpha|\le m} (-1)^{|\alpha|} \int_{\Omega} (D^{\alpha}\vu, \vf_{\alpha}) \, dx.
\end{equation*}
By the uniform ellipticity \eqref{eq11.28}, we get
$$
\delta \int_{\Omega} |D^m\vu|^2 \, dx
\le \int_{\Omega} \Re(a_{\alpha\beta} D^\beta \vu, D^\alpha \vu) \, dx.
$$
Hence, for any $\varepsilon>0$,
\begin{align*}
&\delta \int_{\Omega} |D^m\vu|^2 \, dx+ \lambda \int_{\Omega} |\vu|^2 \, dx
\le \sum_{|\alpha|\le m}(-1)^{|\alpha|} \int_{\Omega}\Re(D^\alpha \vu, \vf_\alpha) \, dx\\
&\le \varepsilon \sum_{|\alpha|\le m}\lambda^{\frac {m-|\alpha|} m} \int_{\Omega} |D^\alpha \vu|^2 \, dx  + N\varepsilon^{-1} \sum_{|\alpha|\le m}\lambda^{-\frac {m-|\alpha|} m}\int_{\Omega} |\vf_\alpha|^2 \, dx.
\end{align*}
To finish the proof, it suffices to use interpolation inequalities
and choose $\varepsilon$ sufficiently small depending on $\delta$, $d$, $m$, and $n$.
\end{proof}

We say that a function $\vu\in W_p(\Omega)$ satisfies \eqref{eq9.15} with the conormal derivative condition on $\Gamma\subset \partial\Omega$ if $u$ satisfies \eqref{eq3.02} for any $\phi\in W^m_q(\Omega)$ which is supported on $\Omega\cup \Gamma$.

By Theorem \ref{theorem08061901} and adapting the proofs of Lemmas 3.2 and 7.2 in \cite{DK10} to the conormal case, we have the following local $L_2$-estimate.

\begin{lemma}
                                            \label{lem6.2}
Let $0<r<R<\infty$. Assume $\vu\in C_{\text{loc}}^\infty(\overline{\bR^{d}_+})$
satisfies
\begin{equation}
                    \label{eq2.54}
\cL_0 \vu=0
\end{equation}
in $B_{R}^+$ with the conormal derivative boundary condition on $\Gamma_R$. Then there exists a constant $N=N(d,m,n,\delta)$ such that for $j=1,\ldots,m$,
$$
\|D^j\vu\|_{L_2(B_r^+)}\leq N(R-r)^{-j}\|\vu\|_{L_2(B_R^+)}.
$$
\end{lemma}

\begin{corollary}
                                    \label{cor6.3}
Let $0<r<R<\infty$ and $a_{\alpha\beta}=a_{\alpha\beta}(x_1)$, $|\alpha|=|\beta|=m$. Assume that $\vu\in C_{\text{loc}}^\infty(\overline{\bR^{d}_+})$ satisfies \eqref{eq2.54}
in $B_R^+$ with the conormal derivative boundary condition on $\Gamma_R$. Then for any multi-index $\theta$ satisfying $\theta_1\le m$ and $|\theta|\ge m$, we have
\begin{equation*}
\|D^\theta\vu\|_{L_2(B_r^+)}\le N\|D^m\vu\|_{L_2(B_R^+)},
\end{equation*}
where $N=N(d,m,n,\delta, R, r, \theta)$.
\end{corollary}
\begin{proof}
It is easily seen that $D_{x'}^{mk} \vu,k=1,2,\ldots,$ also satisfies \eqref{eq2.54} with the conormal derivative boundary condition on $\Gamma_R$. Then by applying Lemma \ref{lem6.2} repeatedly, we obtain
\begin{equation*}
\|D^mD^{mk}_{x'}\vu\|_{L_2(B_{R'}^+)}\le N\|D^m \vu\|_{L_2(B_{R}^+)},
\end{equation*}
where $R'=(r+R)/2$.
From this inequality and the interpolation inequality, we get the desired estimate.
\end{proof}

By using a Sobolev-type inequality, we shall obtain from Corollary \ref{cor6.3} a H\"older estimate of all the $m$-th derivatives of $\vu$ except $D^{\bar\alpha} \vu$, where $\bar\alpha=me_1=(m,0,\ldots,0)$. To compensate this lack of regularity of $D^{\bar\alpha} \vu$, we consider the quantity
$$
\Theta:=\sum_{|\beta|=m}a_{\bar\alpha\beta}D^\beta \vu.
$$

We recall the following useful estimate proved in \cite[Corollary 4.4]{DK10}.

\begin{lemma}
                            \label{corA.2}
Let $k\ge 1$ be an integer, $r\in (0,\infty)$, $p\in [1,\infty]$, $\cD=[0,r]^{d}$, and $u(x)\in L_p(\cD)$. Assume that $D_1^k u=f_0+D_1f_1+\ldots+D_1^{k-1}f_{k-1}$ in $\cD$, where $f_j\in L_p(\cD),j=0,\ldots,k-1$.
Then  $D_1 u \in L_p(\cD)$ and
\begin{equation*}
\|D_1 u\|_{L_p(\cD)}\le N\|u\|_{L_p(\cD)}+N\sum_{j=0}^{k-1}\|f_j\|_{L_p(\cD)},
\end{equation*}
where $N=N(d,k, r)>0$.
\end{lemma}

\begin{corollary}
                \label{lem6.6}
Let $0<r<R<\infty$ and $a_{\alpha\beta}=a_{\alpha\beta}(x_1)$.
Assume $\vu\in C_{\text{loc}}^\infty(\overline{\bR^{d}_+})$ satisfies
\eqref{eq2.54} in $B_R^+$ with the conormal derivative boundary condition on $\Gamma_R$.
Then, for any nonnegative integer $j$,
\begin{equation*}
\|D^j_{x'} \Theta\|_{L_2(B_r^+)}
+ \|D^j_{x'} D_1 \Theta \|_{L_2(B_r^+)}
\le N \|D^m\vu\|_{L_2(B_R^+)},
\end{equation*}
where $N=N(d, m,n,r,R,\delta, j)>0$.
\end{corollary}

\begin{proof}
Due to Corollary \ref{cor6.3} and the fact that $D_{x'}^j\vu$ satisfies \eqref{eq2.54} with the conormal derivative boundary condition, it suffices to prove the desired inequality when $j = 0$ and $R$ is replaced by another $R'$ such that $r < R' < R$.
Obviously, we have
\begin{equation*}
\|\Theta \|_{L_2(B_r^+)}
\le N \|D^m\vu\|_{L_2(B_r^+)}.
\end{equation*}
Thus we prove that, for $R'=(r+R)/2$,
\begin{equation}
                        \label{eq0924}
\|D_1 \Theta \|_{L_2(B_r^+)}
\le N \|D^m\vu\|_{L_2(B_{R'}^+)}.
\end{equation}
From \eqref{eq2.54}, in $B_R^+$ we have
$$
D_1^m\Theta=-\sum_{\substack{|\alpha|=|\beta|=m\\\alpha_1<m}}
D^\alpha(a_{\alpha\beta} D^\beta \vu)
=-\sum_{\substack{|\alpha|=|\beta|=m\\\alpha_1<m}}D_1^{\alpha_1}
(a_{\alpha\beta}D_{x'}^{\alpha'}D^\beta \vu).
$$
Then the estimate \eqref{eq0924} follows from Lemma \ref{corA.2} with a covering argument and Corollary \ref{cor6.3}. The corollary is proved.
\end{proof}

\subsection{H\"older estimates}

By using the $L_2$ estimates obtained in Section \ref{sec3.1}, in this section we shall
derive several H\"older estimates of derivatives of $\vu$.
As usual, for $\mu\in (0,1)$ and a function $u$ defined on $\cD \subset \bR^{d} $, we denote
$$
[u]_{C^{\mu}(\cD)}
= \sup_{\substack{x,y\in\cD\\x\ne y}}\frac{|u(x)-u(y)|}
{|x-y|^{\mu}},
$$
$$
\|u\|_{C^{\mu}(\cD)}=[u]_{C^{\mu}(\cD)}+\|u\|_{L_{\infty}(\cD)}.
$$

\begin{lemma}
                \label{lem6.4}
Let $a_{\alpha\beta}=a_{\alpha\beta}(x_1)$.
Assume that $\vu\in C_{\text{loc}}^\infty(\overline{\bR^{d}_+})$ satisfies
\eqref{eq2.54} in $B_2^+$ with the conormal derivative boundary condition on $\Gamma_2$.
Then for any $\alpha$ satisfying $|\alpha|=m$ and $\alpha_1<m$, we have
\begin{equation*}
\|\Theta\|_{C^{1/2}(B_1^+)}+\|D^\alpha\vu\|_{C^{1/2}(B_1^+)}
\le N \|D^m\vu\|_{L_2(B_2^+)},
\end{equation*}
where $N=N(d,m,n,\delta)>0$.
\end{lemma}
\begin{proof}
The lemma follows from the proof of Lemma 4.1 in \cite{DK10} by using Corollaries \ref{cor6.3} and \ref{lem6.6}.
\end{proof}

For $\lambda\ge0$, let
\begin{equation*}
U=\sum_{|\alpha| \le m}\lambda^{\frac 1 2-\frac {|\alpha|}{2m}} |D^\alpha \vu|,
\quad
U'=\sum_{|\alpha|\le m,\alpha_1<m} \lambda^{\frac 1 2-\frac {|\alpha|} {2m}} |D^\alpha \vu|.
\end{equation*}
Notably, since the matrix $[a_{\bar\alpha\bar\alpha}^{ij}]_{i,j=1}^n$ is positive definite, we have
\begin{equation}
                            \label{eq12.01}
N^{-1}U\le U'+|\Theta|\le NU,
\end{equation}
where $N=N(d,m,n,\delta)$.

\begin{lemma}
                \label{lem6.7}
Let $a_{\alpha\beta}=a_{\alpha\beta}(x_1)$ and $\lambda\ge 0$.
Assume that $\vu\in C_{\text{loc}}^\infty(\overline{\bR^{d}_+})$ satisfies
\begin{equation*}
(-1)^m\cL_0 \vu+\lambda \vu=0
\end{equation*}
in $B_2^+$ with the conormal derivative condition on $\Gamma_2$.
Then we have
\begin{align}
            \label{eq3.11}
\|\Theta\|_{C^{1/2}(B_1^+)}+\|U'\|_{C^{1/2}(B_1^+)}
\le N \|U\|_{L_2(B_2^+)},\\
                \label{eq14.51}
\|U\|_{L_\infty(B_1^+)}\le N\|U\|_{L_2(B_2^+)},
\end{align}
where $N=N(d,m,n,\delta)>0$.
\end{lemma}
\begin{proof}
First we prove \eqref{eq3.11}. The case when $\lambda=0$ follows from Lemma \ref{lem6.4}. To deal with the case $\lambda>0$, we follow an idea by S. Agmon, which was originally used in a quite different situation. Let $
\eta(y)=\cos(\lambda^{1/(2m)}y)+\sin(\lambda^{1/(2m)}y)
$
so that $\eta$ satisfies
$$
D^{2m}\eta=(-1)^m\lambda \eta,\quad
\eta(0)=1,\quad |D^j\eta(0)|=\lambda^{j/(2m)},\,\,\,j=1,2,\ldots.
$$
Let $z = (x,y)$ be a point in $\bR^{d+1}$, where $x \in \bR^{d}$, $y \in \bR$,
and $\hat{\vu}(z)$ and $\hat{B}_r^+$ be given by
$$
\hat{\vu}(z) = \hat{\vu}(x,y) = \vu(x)\eta(y),
\quad
\hat{B}_r^+ = \{ |z| < r: z \in \bR^{d+1},x_1>0 \}.
$$
Also define
$$
\hat\Theta=\sum_{|\beta|=m}a_{\bar\alpha\beta}D^{(\beta,0)} \hat\vu.
$$
It is easily seen that $\hat{\vu}$ satisfies
$$
(-1)^m\cL_0\hat{\vu}+(-1)^mD^{2m}_y \hat{\vu} = 0
$$
in $\hat{B}_2^+$ with the conormal derivative condition on $\hat{B}_2\cap \partial \bR^{d+1}_+$.
By Lemma \ref{lem6.4} applied to $\hat{\vu}$ we have
\begin{equation}                                 \label{eq0804}
 \|\hat\Theta\|_{C^{1/2}(\hat{B}_1^+)}+\big\| D_z^\beta\hat{\vu}\big\|_{C^{1/2}(\hat{B}_1^+)}
\le N(d,m,n,\delta) \|D^m_z\hat{\vu}\|_{L_2(\hat{B}_2^+)}
\end{equation}
for any $\beta=(\beta_1,\ldots,\beta_{d+1})$ satisfying $|\beta|=m$ and $\beta_1<m$.
Notice that for any $\alpha=(\alpha_1,\ldots,\alpha_d)$ satisfying $|\alpha|\le m$ and $\alpha_1<m$,
$$
\lambda^{\frac 1 2-\frac {|\alpha|} {2m}} \big\|D^\alpha\vu\big\|_{C^{1/2}(B_1^+)}
\le N\big\| D^\beta_z\hat{\vu} \big\|_{C^{1/2}(\hat{B}_1^+)},\quad \beta=(\alpha_1,\ldots,\alpha_d,m-|\alpha|),
$$
$$
\|\Theta\|_{C^{1/2}(B_1^+)}\le \|\hat\Theta\|_{C^{1/2}(\hat{B}_1^+)},
$$
and $D_z^m \hat{\vu}$ is a linear combination of
$$
\lambda^{\frac 1 2-\frac k {2m}}\cos( \lambda^{\frac 1 {2m}} y) D^k_x\vu,
\quad
\lambda^{\frac 1 2-\frac k {2m}}\sin( \lambda^{\frac 1 {2m}} y) D^k_x\vu,\quad k=0,1,\ldots,m.
$$
Thus  the right-hand side of \eqref{eq0804} is less than the right-hand side of \eqref{eq3.11}. This completes the proof of \eqref{eq3.11}. Finally, we get \eqref{eq14.51} from \eqref{eq3.11}  and \eqref{eq12.01}.
\end{proof}

Similarly, we have the following interior estimate.

\begin{lemma}
                \label{cor3.5}
Let $a_{\alpha\beta}=a_{\alpha\beta}(x_1)$ and $\lambda\ge0$.
Assume that $\vu\in C_{\text{loc}}^\infty(\bR^{d})$ satisfies
\begin{equation*}
(-1)^m\cL_0 \vu+\lambda \vu=0
\end{equation*}
in $B_2$.
Then we have
\begin{align*}
\|\Theta\|_{C^{1/2}(B_1)}+\left\|U'\right\|_{C^{1/2}(B_1)}
\le N \|U\|_{L_2(B_2)},\\
\|U\|_{L_\infty(B_1)}\le N\|U\|_{L_2(B_2)},
\end{align*}
where $N=N(d,m,n,\delta)>0$.
\end{lemma}

\subsection{The maximal function theorem and a generalized Fefferman-Stein theorem}
We recall the maximal function theorem and  a generalized Fefferman-Stein theorem.
Let
$$
\cQ=\{B_r(x): x \in \bR^{d}, r \in (0,\infty)\}.
$$
For a function $g$ defined in $\bR^{d}$,
the maximal function of $g$ is given by
$$
\cM  g (x) = \sup_{B \in \cQ, x \in B} \dashint_{B} |g(y)| \, dy.
$$
By the Hardy--Littlewood maximal function theorem,
$$
\| \cM  g \|_{L_p(\bR^{d})} \le N \| g\|_{L_p(\bR^{d})},
$$
if $g \in L_p(\bR^{d})$, where $1 < p < \infty$ and $N = N(d,p)$.

Theorem \ref{th081201} below is from \cite{Krylov08}
and can be considered as a generalized version of the Fefferman-Stein Theorem.
To state the theorem,
let
$$
\bC_l = \{ C_l(i_1, \ldots, i_d), i_1, \ldots, i_d \in \bZ, i_1\ge 0 \},
\quad l \in \bZ
$$
be the collection of
partitions given by dyadic cubes in $\bR^{d}_+$
\begin{equation*}
 [ i_1 2^{-l}, (i_1+1)2^{-l} ) \times \ldots \times [ i_d 2^{-l}, (i_d+1)2^{-l} ).
\end{equation*}

\begin{theorem}                         \label{th081201}
Let $p \in (1, \infty)$, and $U,V,F\in L_{1,\text{loc}}(\bR^{d}_+)$.
Assume that we have $|U| \le V$
and, for each $l \in \bZ$ and $C \in \bC_l$,
there exists a measurable function $U^C$ on $C$
such that $|U| \le U^C \le V$ on $C$ and
\begin{equation*}                            
\int_C |U^C - \left(U^C\right)_C| \,dx
\le \int_C F(x) \,dx.
\end{equation*}
Then
$$
\| U \|_{L_p(\bR^{d}_+)}^p
\le N(d,p) \|F\|_{L_p(\bR^{d}_+)}\| V \|_{L_p(\bR^{d}_+)}^{p-1}.
$$
\end{theorem}

\subsection{Approximations of Reifenberg domains}

Let $\Omega$ be a domain in $\bR^d$. Throughout this subsection, we assume that, for any $x\in \partial\Omega$ and $r\in (0,1]$, $\Omega$ satisfies \eqref{eq3.49} in an appropriate coordinate system.
That is, $\Omega$ satisfies the following assumption with $\gamma<1/50$.

\begin{assumption}[$\gamma$]
                                        \label{assump11}
There is a constant $R_0\in (0,1]$ such that the following holds.
For any $x\in \partial\Omega$ and $r\in (0,R_0]$, there is a coordinate system depending on $x$ and $r$
such that in this new coordinate system, we have
\begin{equation}
                    \label{eq1005_1}
 \{(y_1,y'):x_1+\gamma r<y_1\}\cap B_r(x)
 \subset\Omega_r(x)
 \subset \{(y_1,y'):x_1-\gamma r<y_1\}\cap B_r(x).
\end{equation}
\end{assumption}

For any $\varepsilon\in (0,1)$, we define
\begin{equation}
                                        \label{eq2.22}
\Omega^\varepsilon=\{x\in \Omega\,|\,\dist(x,\partial\Omega)>\varepsilon\}.
\end{equation}

We say that a domain is a Lipschitz domain if locally the boundary is the graph of a Lipschitz function in some coordinate system. More precisely,

\begin{assumption}[$\theta$]
                                    \label{assump3}
There is a constant $R_1\in (0,1]$ such that, for any $x\in \partial\Omega$ and $r\in(0,R_1]$, there exists a Lipschitz
function $\phi$: $\bR^{d-1}\to \bR$ such that
$$
\Omega\cap B_r(x_0) = \{x \in B_r(x_0)\, :\, x_1 >\phi(x')\}
$$
and
$$
\sup_{x',y'\in B_r'(x_0'),x' \neq y'}\frac {|\phi(y')-\phi(x')|}{|y'-x'|}\le \theta
$$
in some coordinate system.
\end{assumption}

We note that if $\Omega$ satisfies Assumption \ref{assump3} ($\theta$) with a constant $R_1$, then $\Omega$ satisfies Assumption \ref{assump11} with $R_1$ and $\theta$ in place of $R_0$ and $\gamma$, respectively.

Next we show that $\Omega^\varepsilon$ is a Lipschitz domain and Reifenberg flat with uniform parameters if $\Omega$ is Reifenberg flat. A related result was proved in \cite{BW07} which, in our opinion, contains a flaw.

\begin{lemma}
                            \label{lem3.11}
Let $\Omega$ satisfy Assumption \ref{assump11} ($\gamma$).
Then for any $\varepsilon\in (0,R_0/4)$, $\Omega^\varepsilon$ satisfies Assumption \ref{assump11} ($N_0\gamma^{1/2}$) with $R_0/2$ in place of $R_0$, and satisfies Assumption \ref{assump3} ($N_0\gamma^{1/2}$) with $R_1=\varepsilon$.
Here $N_0$ is a universal constant.
\end{lemma}

\begin{proof}
We first prove that $\Omega^\varepsilon$ satisfies Assumption \ref{assump3} ($N_0\gamma^{1/2}$) with $R_1 = \varepsilon >0$.
In particular, we show that, for each $x_0 \in \partial\Omega^{\varepsilon}$, there exists a function $\phi:\bR^{d-1} \to \bR$
such that
\begin{equation}
							\label{eq0930}
\Omega^{\varepsilon} \cap B_{\varepsilon}(x_0)
= \{ x \in B_{\varepsilon}(x_0) : x_1 > \phi(x') \},
\quad
\frac{|\phi(y') - \phi(x')|}{|x'-y'|} \le N_0 \gamma^{1/2}
\end{equation}
for all $x',y' \in B'_{\varepsilon}(x_0')$, $x' \ne y'$.
Indeed, this implies Assumption \ref{assump3} ($N_0\gamma^{1/2}$) since for a fixed $x_0 \in \partial \Omega^\varepsilon$ we can use the same $\phi$ for all $r \in (0,\varepsilon)$.

Let $0$ be a point on $\partial \Omega$ such that $|x_0 - 0| = \varepsilon$.
That is, we have a coordinate system and $r_0 := 4 \varepsilon  < R_0$ such that $\partial \Omega \cap B_{r_0}(0)$ is trapped between
$\{x_1 = \gamma r_0\}$ and $\{x_1 = - \gamma r_0\}$. See Figure \ref{fg3}.
Note that $B_{\varepsilon}(x_0) \subset B_{r_0}(0)$
since, for $x \in B_{\varepsilon}(x_0)$,
$$
|x| \le |x-x_0| + |x_0| < 2\varepsilon < r_0 = 4\varepsilon.
$$
We show that for any $y, z \in \partial \Omega^\varepsilon \cap B_{\varepsilon}(x_0)$
\begin{equation}
							\label{eq1003_1}
|y_1 - z_1| \le N_0 \gamma^{1/2} |y'-z'|,
\end{equation}
which implies \eqref{eq0930}.
For $y, z \in \partial \Omega^\varepsilon \cap B_{\varepsilon}(x_0)$,
we see that
\begin{equation}
							\label{eq1003_2}
\varepsilon - \gamma r_0 < y_1 < \varepsilon + \gamma r_0,
\quad
\varepsilon - \gamma r_0 < z_1 < \varepsilon + \gamma r_0.
\end{equation}
Without loss of generality we assume that $y_1 \ge z_1$.
\begin{figure}[b]
\centering
\includegraphics{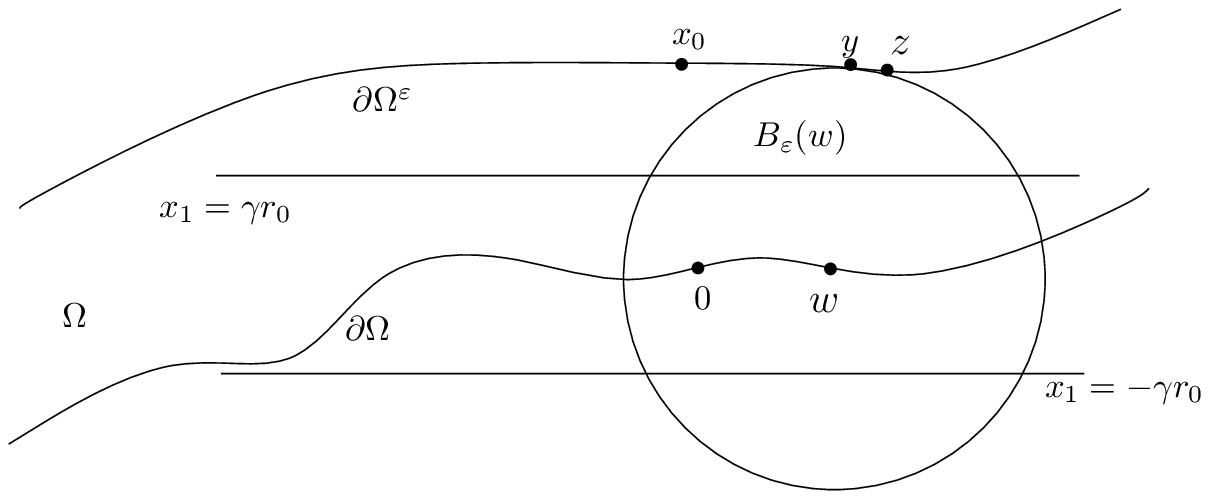}
\caption{}\label{fg3}
\end{figure}
To prove \eqref{eq1003_1}, let us consider two cases.
First, let $\varepsilon \gamma^{1/2} \le |y'-z'|$.
In this case, due to the inequalities \eqref{eq1003_2}, we have
$$
\frac{|y_1-z_1|}{|y'-z'|} \le \frac{2\gamma r_0}{\varepsilon \gamma^{1/2}} = 8 \gamma^{1/2},
$$
which proves \eqref{eq1003_1}.

Now let $|y'-z'| \le \varepsilon \gamma^{1/2}$.
In this case, find $w \in \partial \Omega$ such that $|y-w| = \varepsilon$.
Note that $B_\varepsilon(w) \subset B_{r_0}(0)$ since
$$
|w| \le |w-y| + |y-x_0| + |x_0| < 3 \varepsilon < r_0 = 4\varepsilon.
$$
We estimate $|w'-z'|$ as follows.
Using the fact that
$-\gamma r_0 < w_1 < \gamma r_0$
and the first inequality in \eqref{eq1003_2},
we have
$$
|y_1-w_1| \ge \varepsilon - 2 \gamma r_0 >0.
$$
Thus using the equality
$$
|w'-y'|^2 + |w_1-y_1|^2 = \varepsilon^2,
$$
we see that
$$
|w'-y'|^2 \le \varepsilon^2 - (\varepsilon - 2\gamma r_0)^2
\le 4 \varepsilon \gamma r_0 = 4^2 \varepsilon^2 \gamma.
$$
Hence
$$
|w'-z'| \le |w'-y'| + |y'-z'| <
5\varepsilon \gamma^{1/2}.
$$
Since $y_1 \ge z_1$, $|w'-z'| \le 5\varepsilon \gamma^{1/2}$,
and $z$ is above the ball $B_\varepsilon(\omega)$ (recall that $\gamma<1/50$),
it follows that
$$
\frac{y_1 - z_1}{|y'-z'|}
\le \frac{d}{dx} \left(-\sqrt{\varepsilon^2 - x^2}\right) \bigg|_{x=5 \varepsilon \gamma^{1/2}}
\le N_0 \gamma^{1/2}.
$$
Thus \eqref{eq1003_1} is proved.
Therefore, we have proved that $\Omega^\varepsilon$ satisfies Assumption \ref{assump3} ($N_0\gamma^{1/2}$) with $R_1 = \varepsilon$.
As pointed out earlier, this shows that $\Omega^\varepsilon$ satisfies \eqref{eq1005_1} for all $0 < r < \varepsilon$.
Thus in order to completely prove that $\Omega^\varepsilon$ satisfies Assumption \ref{assump11} ($N_0\gamma^{1/2}$) with $R_0/2$, we need to prove that
$\Omega^\varepsilon$ satisfies \eqref{eq1005_1} for $\varepsilon \le r < R_0/2$.

Let $\varepsilon \le r < R_0/2$ and $x_0 \in \partial\Omega^{\varepsilon}$. Find $0 \in \partial\Omega$
such that $|x_0 - 0| = \varepsilon$.
Then
$$
B_r(x_0) \subset B_R(0),
$$
where
$R = \varepsilon+r < R_0$.
Then the first coordinate $x_1$ of the point $x \in \partial \Omega^{\varepsilon} \cap B_r(x_0)$ is trapped by
$$
\varepsilon - \gamma R < x_1 < \varepsilon + \gamma R,
$$
which is the same as
$$
\varepsilon - \gamma ( \varepsilon+r) < x_1 < \varepsilon + \gamma ( \varepsilon+r).
$$
Note that
$$
\gamma (\varepsilon + r) \le 2 \gamma r \le 2 \gamma^{1/2}r.
$$
Thus each $x_1$ of $x \in \partial\Omega^\varepsilon \cap B_r(x_0)$ satisfies
\eqref{eq1005_1} with $2\gamma^{1/2}$ in place of $\gamma$.
The lemma is proved.
\end{proof}

The next approximation result is well known. See, for instance, \cite{Lie85}.

\begin{lemma}
                                \label{lem3.19}
Let $\Omega$ be a domain in $\bR^d$ and satisfy Assumption \ref{assump3} ($\theta$) with some $\theta>0$ and $R_1\in (0,1]$. Then there exists a sequence of expanding smooth subdomains $\Omega^k,k=1,2,\ldots$, such that $\Omega^k\to \Omega$ as $k\to \infty$ and each $\Omega^k$ satisfies Assumption \ref{assump3} ($N_0\theta$) with $R_1/2$ in place of $R_1$. Here $N_0$ is a universal constant.
\end{lemma}

\section{Systems on a half space}
                            \label{sec4}

\subsection{Estimates of mean oscillations}
Now we prove the following estimate of mean oscillations.
As in Section \ref{sec_aux}, we assume that all the lower-order coefficients of $\cL$ are zero. For $\vf_\alpha= (f_\alpha^1, \ldots, f_\alpha^n)^{\text{tr}}$, we denote
$$
F=\sum_{|\alpha|\le m}\lambda^{\frac {|\alpha|} {2m}-\frac 1 2}|\vf_\alpha|.
$$

\begin{proposition}
                \label{prop7.9}
Let $x_0\in \overline{\bR^d_+}$, $\gamma \in (0,1/4)$, $r\in (0,\infty)$, $\kappa\in [64,\infty)$, $\lambda\ge 0$, $\nu \in (2,\infty)$, $\nu'=2\nu/(\nu-2)$, and
$\vf_\alpha= (f_\alpha^1, \ldots, f_\alpha^n)^{\text{tr}} \in L_{2,\text{loc}}(\overline{\bR^{d}_+})$.
Assume that $\kappa r\le R_0$ and $\vu\in W_{\nu,\text{loc}}^m(\overline{\bR^d_+})$ satisfies
\begin{equation}
                    \label{eq11.13}
(-1)^m\cL \vu+\lambda \vu=\sum_{|\alpha|\le m}D^\alpha \vf_\alpha
\end{equation}
in $B^+_{\kappa r}(x_0)$ with the conormal derivative condition on $\Gamma_{\kappa r}(x_0)$. Then under Assumption \ref{assumption20100901} ($\gamma$),
there exists a function $U^B$ depending on
$B^+:=B^+_{\kappa r}(x_0)$ such that $N^{-1}U\le U^B\le NU$ and
$$
\big(|U^B-(U^B)_{B_r^+(x_0)}|\big)_{B_r^+(x_0)}
\le N(\kappa^{-1/2}+(\kappa \gamma)^{1/2} \kappa^{d/2}) \big(U^2\big)_{B^+_{\kappa r}(x_0)}^{1/2}
$$
\begin{equation}
            \label{eq5.42}
+N\kappa^{d/2}\left[(F^2)_{B^+_{\kappa r}(x_0)}^{1/2}+\gamma^{1/\nu'}
(U^\nu)_{B^+_{\kappa r}(x_0)}^{1/\nu}\right],
\end{equation}
where $N=N(d,m,n,\delta,\nu)>0$.
\end{proposition}

The proof of the proposition is split into two cases.

{\em Case 1: the first coordinate of $x_0$ $\ge \kappa r/16$.} In this case, we have
$$
B_r^+(x_0)=B_r(x_0)\subset B_{\kappa r/16}(x_0)\subset \bR^{d}_+.
$$
With $B_{\kappa r/16}$ in place of $B^+_{\kappa r}$ in the right-hand side of \eqref{eq5.42}, the problem is reduced to an interior mean oscillation estimate. Thus the proof can be done in the same way as in Proposition 7.10 in \cite{DK10} using Theorem \ref{theorem08061901} and Lemma \ref{cor3.5}.

{\em Case 2: $0\le$ the first coordinate of $x_0$ $< \kappa r/16$.}
Notice that in this case,
\begin{equation}
                                \label{eq22.15}
B_r^+(x_0)\subset B^+_{\kappa r/8}(\hat x_0) \subset B^+_{\kappa r/4}(\hat x_0)\subset B^+_{\kappa r/2}(\hat x_0)
\subset B^+_{\kappa r}(x_0),
\end{equation}
where $\hat x_0:=(0,x_0')$.
Denote $R=\kappa r/2(< R_0)$. Because of Assumption \ref{assumption20100901}, after a linear transformation, which is an orthogonal transformation determined by $B=B_R(\hat x_0)$ followed by a translation downward,  we may assume
\begin{equation}
                    \label{eq17.29}
B_R^+(\hat{y}_0)
 \subset\Omega_{R}(\hat{y}_0)
 \subset \{(y_1,y'):-2\gamma R< y_1\}\cap B_{R}(\hat{y}_0)
\end{equation}
and
\begin{equation}
                                \label{eq17_50}
\sup_{|\alpha|=|\beta|=m}\int_{B_R(\hat{y}_0)} |a_{\alpha\beta}(x) - \bar{a}_{\alpha\beta}(y_1)| \, dy \le \gamma |B_R|.
\end{equation}
\begin{figure}[h]
\centering
\includegraphics{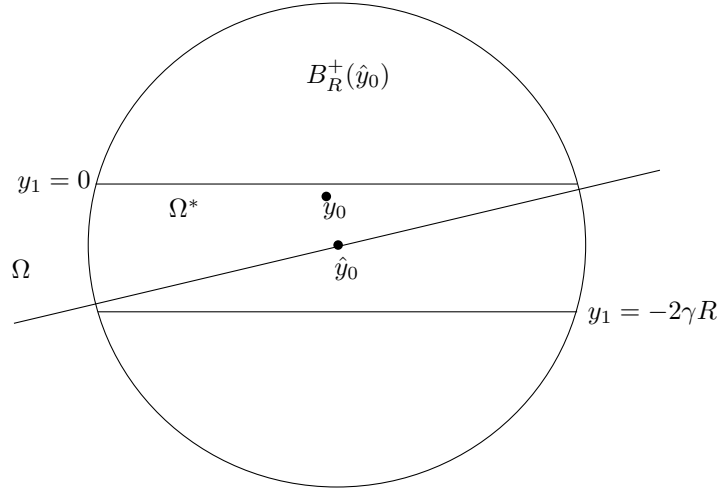}
\caption{$\hat y_0 $ (or $y_0$) is the new coordinates of $\hat x_0$ (or $x_0$).}\label{fg1}
\end{figure}
Here $\Omega$ is the image of $\bR^d_+$ under the linear transformation
and $\hat y_0 $ (or $y_0$) is the new coordinates of $\hat x_0$ (or $x_0$). See Figure \ref{fg1}.
Then \eqref{eq22.15} becomes
\begin{equation}
                                \label{eq22.16}
\Omega_r(y_0)\subset \Omega_{R/4}(\hat y_0) \subset \Omega_{R/2}(\hat y_0)
\subset \Omega_{R}(\hat y_0)
\subset \Omega_{\kappa r}(y_0).
\end{equation}

For convenience of notation, in the new coordinate system we still denote the corresponding unknown function, the coefficients, and the data by $\vu$, $a_{\alpha\beta}$, $\bar{a}_{\alpha\beta}$, and $\vf_\alpha$, respectively.
Note that, without loss of generality, we may assume that the coefficients $\bar{a}_{\alpha\beta}(y_1)$ in \eqref{eq17_50} are infinitely differentiable.

Below we present a few lemmas, which should be read as parts of the proof of the second case.

Let us introduce the following well-known extension operator.
Let $\{c_1, \cdots, c_{m}\}$ be the solution to the system:
\begin{equation}							\label{eq0514-02}
\sum_{k=1}^{m} \left(-\frac{1}{k}\right)^j c_k = 1,
\quad
j=0,\cdots,m-1.	
\end{equation}
For a function $w$ defined on $\bR^d_+$, set
\begin{equation*}
\cE_{m} w =
\left\{
\begin{aligned}
&w(y_1,y')	\quad \text{if} \quad y_1 > 0\\
&\sum_{k=1}^{m} c_k w(-\frac{1}{k}y_1,y') \quad \text{otherwise}
\end{aligned}
\right..
\end{equation*}
Note that
$\cE_{m} w \in W^{m}_{2,\text{loc}}(\bR^d)$ if $w \in W^{m}_{2,\text{loc}}(\overline{\bR^d_+})$.
Indeed, by \eqref{eq0514-02}
$$
D_1^j \left(\sum_{k=1}^{m} c_k w(-\frac{1}{k}y_1,y')\right)\bigg|_{y_1=0}
= \sum_{k=1}^{m} \left(-\frac1k\right)^j c_k D_1^jw(0,y')
= D_1^jw(0,y')
$$
for $j = 0, \cdots, m-1$.

Denote $\Omega^*=\bR^d_-\cap\Omega\cap B_R(\hat y_0)$. Recall that in the new coordinate system we still denote the corresponding unknown function, the coefficients, and the data by $\vu$, $a_{\alpha\beta}$, $\bar{a}_{\alpha\beta}$, and $\vf_\alpha$, respectively. Throughout the end of this subsection, the derivatives are taken with respect to the $y$-coordinates.
The following lemma contains the key observation in the proof of Proposition \ref{prop7.9}.

\begin{lemma}
                            \label{lem4.2}
The function $\vu$ satisfies
\begin{align}
                                    \label{eq17.23b}
(-1)^m \cL_0 \vu+\lambda \vu&=(-1)^m \sum_{|\alpha|=|\beta|=m}D^\alpha\left(
(\bar a_{\alpha\beta}- a_{\alpha\beta})D^\beta \vu\right)\\
&\quad + \sum_{|\alpha|\le m} D^\alpha \tilde \vf_\alpha+\sum_{|\alpha|=m}D^\alpha \vg_\alpha-\lambda \vh\nonumber
\end{align}
in $B_R^+(\hat y_0)$ with the conormal derivative boundary condition on $\Gamma_R(\hat y_0)$.
In the above, $\cL_0$ is the differential operator with the coefficients $\bar a_{\alpha\beta}$ from \eqref{eq17_50}, and
\begin{align*}
\tilde \vf_\alpha&=\vf_\alpha+c_{\alpha,k}
\vf_\alpha(-ky_1,y')\,1_{(-ky_1,y')\in \Omega^*},\\
\vg_\alpha&=c_{\alpha,k}(-1)^{m+1}\sum_{|\beta|=m}\sum_{k=1}^m
a_{\alpha\beta}(-ky_1,y') (D^{\beta}\vu)(-ky_1,y')\,1_{(-ky_1,y')\in \Omega^*},\\
\vh&=\sum_{k=1}^m
kc_k\vu(-ky_1,y')\, 1_{(-ky_1,y')\in \Omega^*},
\end{align*}
where
$c_{\alpha,k}=(-1)^{\alpha_1}c_kk^{-\alpha_1+1}$ are constants.
\end{lemma}
\begin{proof}
Take a test function $\phi=(\phi^1,\phi^2,\ldots,\phi^n)\in W^m_2(B_R^+(\hat y_0))$ which vanishes near $\bR^d_+\cap \partial B_R(\hat y_0)$. Due to \eqref{eq17.29}, it is easily seen that $\cE_m \phi\in W^m_2(\Omega_R(\hat y_0))$ and vanishes near $\Omega\cap \partial B_R(\hat y_0)$. Since $\vu$ satisfies \eqref{eq11.13} with the conormal derivative condition on  $\partial\Omega\cap B_R(\hat y_0)$, we have
$$
\int_{\Omega_R(\hat y_0)}\sum_{|\alpha|=|\beta|=m}D^\alpha\cE_m\phi \cdot
a_{\alpha\beta}D^\beta\vu +\lambda \cE_m\phi\cdot \vu\,dy
$$
$$
=\sum_{|\alpha|\le m}\int_{\Omega_R(\hat y_0)}(-1)^{|\alpha|}D^\alpha\cE_m\phi\cdot
\vf_\alpha\, dy.
$$
From this identity and the definition of the extension operator $\cE_m$, a straightforward calculation gives
$$
\int_{B_R^+(\hat y_0)}\sum_{|\alpha|=|\beta|=m}D^\alpha\phi \cdot
\bar a_{\alpha\beta}D^\beta\vu +\lambda \phi\cdot \vu\,dy
$$
$$
=\sum_{|\alpha|=|\beta|=m}
\int_{B_R^+(\hat y_0)}D^\alpha\phi \cdot
(\bar a_{\alpha\beta}-a_{\alpha\beta})D^\beta\vu \, dy
+\sum_{|\alpha|\le m}\int_{B_R^+(\hat y_0)}(-1)^{|\alpha|}D^\alpha\phi\cdot
\tilde\vf_\alpha\,dy
$$
$$
=\sum_{|\alpha|=m}
\int_{B_R^+(\hat y_0)}D^\alpha\phi \cdot
(-1)^{|\alpha|}\vg_\alpha \,dy
-\lambda \int_{B_R^+(\hat y_0)}\phi\cdot\vh\,dy.
$$
The lemma is proved.
\end{proof}

Set
$$
G_\alpha = (-1)^m \sum_{|\beta|=m} \left(\bar a_{\alpha\beta} - a_{\alpha\beta}\right)D^\beta u + \tilde{\vf}_\alpha + \vg_\alpha
\quad
\text{for}
\quad
|\alpha|=m,
$$
$$
G_\alpha = \tilde{\vf}_\alpha
\quad
\text{for}
\quad
0\le |\alpha| < m.
$$
We see that $G_\alpha \in L_2(B_R^+(\hat{y}_0))$, and by \eqref{eq17.23b}
$$
(-1)^m \cL_0 \vu + \lambda \vu = \sum_{|\alpha| \le m} D^{\alpha} G_\alpha - \lambda \vh.
$$
Take $\varphi$ to be an infinitely differentiable function such that
$$
0 \le \varphi \le 1,
\quad
\varphi = 1 \,\, \text{on} \,\, B_{R/2}(\hat y_0),
\quad
\varphi = 0 \,\, \text{outside} \,\, B_R(\hat y_0).
$$
Then we find a unique solution $\vw \in W_2^m(\bR^d_+)$ satisfying
\begin{equation}
							\label{eq08_01}
(-1)^m \cL_0 \vw + \lambda \vw = \sum_{|\alpha| \le m} D^{\alpha} (\varphi G_\alpha) - \lambda \varphi\vh
\end{equation}
with the conormal derivative condition on $\partial \bR^d_+$.
By Theorem \ref{theorem08061901} we have
\begin{equation}
							\label{eq08_02}
\sum_{|\alpha|\le m}\lambda^{\frac12-\frac {|\alpha|} {2m}} \|D^\alpha \vw \|_{L_2(\bR^d_+)}
\le  N \sum_{|\alpha|\le m}\lambda^{\frac {|\alpha|} {2m}-\frac12} \|\varphi G_\alpha\|_{L_2(\bR^d_+)}
+ N \lambda \|\varphi \vh\|_{L_2(\bR^d_+)}.
\end{equation}
%
Now we set $\vv := \vu-\vw$ in $B_R(\hat y_0) \cap \overline{\bR^d_+}$.
Then $\vv$ satisfies
\begin{equation}
							\label{eq08_03}
(-1)^m \cL_0 \vv + \lambda \vv = 0
\end{equation}
in $B_{R/2}^+(\hat y_0)$ with the conormal derivative condition on $\Gamma_{R/2}(\hat y_0)$.
Since the coefficients of $\cL_0$ are infinitely differentiable, by the classical theory $\vv$ is infinitely differentiable in $B_{R/2}(\hat{y}_0) \cap \overline{\bR^d_+}$.

Recall that
$$
U=\sum_{|\alpha| \le m}\lambda^{\frac 1 2-\frac {|\alpha|}{2m}} |D^\alpha \vu|,
\quad
F=\sum_{|\alpha|\le m}\lambda^{\frac {|\alpha|} {2m}-\frac 1 2}|\vf_\alpha|.
$$

\begin{lemma}
We have
\begin{equation}
            \label{eq21.52h}
\sum_{k=0}^m\lambda^{\frac 1 2-\frac k {2m}}(|D^k  \vw|^2)_{B_R^+(\hat y_0)}^{1/2}
\le N\gamma^{1/ {\nu'}} (U^\nu)_{\Omega_R(\hat y_0)}^{1/ \nu}+
N(F^2)_{\Omega_R(\hat y_0)}^{1/2},
\end{equation}
where $\nu$ and $\nu'$ are from Proposition \ref{prop7.9}.
\end{lemma}

\begin{proof}
By \eqref{eq08_02} and the definition of $G_\alpha$, we have
\begin{multline*}
\sum_{|\alpha|\le m}\lambda^{\frac 1 2-\frac {|\alpha|} {2m}}\|D^{\alpha}  \vw\|_{L_2(\bR^d_+)}
\le N \sum_{|\alpha|=|\beta|=m}\|\varphi(\bar a_{\alpha\beta}- a_{\alpha\beta})D^\beta \vu\|_{L_2(\bR^d_+)}
\\
+N\sum_{|\alpha|\le m}\lambda^{\frac{|\alpha|}{2m}-\frac 1 2} \|\varphi\tilde \vf_\alpha\|_{L_2(\bR^d_+)}
+N \sum_{|\beta|=m}\|\varphi \vg_\alpha\|_{L_2(\bR^d_+)}
+N \lambda \|\varphi \vh\|_{L_2(\bR^d_+)}.
\end{multline*}

Note that $\Omega^*$ lies in the strip $B_R(\hat y_0)\cap \{y:  -2\gamma R<y_1<0\}$.
Thus, by the definitions of $\tilde \vf_\alpha$, $\vg_\alpha$, and $\vh$,
it follows that the left-hand side  of \eqref{eq21.52h} is less than a constant times
\begin{align*}
\sum_{|\alpha|=|\beta|=m}&\left(|(\bar a_{\alpha\beta}- a_{\alpha\beta})D^\beta \vu|^2\right)^{1/2}_{B_R^+(\hat y_0)}
+ \sum_{|\alpha|\le m}\lambda^{\frac{|\alpha|}{2m}-\frac 1 2} \left(|\vf_\alpha|^2\right)^{1/2}_{\Omega_R(\hat y_0)}\\
&+ \left(I_{\{-2\gamma R<y_1<0\}}|D^{m}\vu|^2\right)^{1/2}_{\Omega_R(\hat y_0)}+\lambda
\left(I_{\{-2\gamma R<y_1<0\}}|\vu|^2\right)^{1/2}_{\Omega_R(\hat y_0)}\\
&:=I_1+I_2+I_3+I_4.
\end{align*}
By using \eqref{eq17_50} and H\"older's inequality, we see that
$$
I_1\le N\gamma^{1/ {\nu'}} (U^\nu)_{\Omega_R(\hat y_0)}^{1/ \nu}.
$$
It is clear that $I_2$ is bounded by $N(F^2)_{\Omega_R(\hat y_0)}^{1/2}$.
Observe that by H\"older's inequality we have
\begin{align*}
\left(I_{\{-2\gamma R < y_1 < 0\}}
|D^{m}\vu|^2\right)^{1/2}_{\Omega_R(\hat y_0)}
&\le \left(I_{\{-2\gamma R < y_1 < 0\}}\right)_{\Omega_R(\hat y_0)}^{1/\nu'}
\left(|D^m\vu|^\nu\right)^{1/\nu}_{\Omega_R(\hat y_0)}
\\
&\le N\gamma^{1/\nu'}(|D^m\vu|^{\nu})_{\Omega_R(\hat y_0)}^{1/\nu}.
\end{align*}
Thus $I_3$ is also bounded by $N\gamma^{1/ {\nu'}} (U^\nu)_{\Omega_R(\hat y_0)}^{1/ \nu}$. In a similar way, $I_4$ is bounded by $N\gamma^{1/ {\nu'}} (U^\nu)_{\Omega_R(\hat y_0)}^{1/ \nu}$.
Therefore, we conclude \eqref{eq21.52h}.
\end{proof}

Now we are ready to complete the proof of Proposition \ref{prop7.9}.
We shall show that $U^B:=U'+|\Theta|$ satisfies the inequalities in the proposition.
First, we consider the case when $\kappa \gamma \le 1/10$.
By \eqref{eq21.52h}, it follows that
\begin{equation}
            \label{eq18.34h}
(W^2)_{B_{R}^+(\hat y_0)}^{1/2}
\le N\gamma^{1/ {\nu'}} (U^\nu)_{\Omega_{R}(\hat y_0)}^{1/ \nu}+
N(F^2)_{\Omega_{R}(\hat y_0)}^{1/2}.
\end{equation}
Noting that $B_r^+(y_0) \subset B_R^+(\hat y_0)$ and, since $\kappa \gamma\le 1/10$, $|B_R^+(\hat y_0)|/|B_r^+(y_0)|
\le N(d) \kappa^{d}$, we obtain from \eqref{eq18.34h}
\begin{equation}
                                \label{eq28_01}
(W^2)_{B_r^+(y_0)}^{1/2}
\le N\kappa^{\frac d 2} \left(\gamma^{1/ {\nu'}} (U^\nu)_{\Omega_{R}(\hat y_0)}^{1/ \nu}+
(F^2)_{\Omega_{R}(\hat y_0)}^{1/2}\right).
\end{equation}
Next we denote
$$
\cD_1=\Omega_{r}(y_0)\cap \{y_1<0\}=\Omega^*,
\quad
\cD_2=B_{r}^+(y_0),
\quad
\cD_3=B_{R/4}^+(\hat y_0).
$$
Since $\vv$ in \eqref{eq08_03} is infinitely differentiable,  by applying Lemma \ref{lem6.7} to the system \eqref{eq08_03} with a scaling argument, we compute
\begin{multline}
                                \label{eq23.10}
\big(|V'-(V')_{\cD_2}|\big)_{\cD_2}
+\big(|\hat \Theta-(\hat \Theta)_{\cD_2}|\big)_{\cD_2}
\le N r^{1/2}\big([V']_{C^{1/2}(\cD_2)}
+[\hat \Theta]_{C^{1/2}( \cD_2)}\big)\\
\le N r^{1/2}\big([V']_{C^{1/2}(\cD_3)}
+[\hat \Theta]_{C^{1/2}( \cD_3)}\big)
\le N\kappa^{-1/2}(V^2)_{B^+_{R/2}(\hat y_0)}^{1/2}.
\end{multline}
Thanks to the fact that $\kappa\gamma\le 1/10$, we have
\begin{equation}
                                \label{eq21.31}
|\cD_1|\le N\kappa\gamma |\cD_2|,\quad
|\Omega_R(\hat y_0)|\le N\kappa^d |\cD_2|.
\end{equation}
By combining \eqref{eq18.34h}, \eqref{eq28_01}, and \eqref{eq23.10}, we get
\begin{align*}
\big(|&U^{B}-(U^B)_{\cD_2}|\big)_{\cD_2}\nonumber\\
&\le N\big(|V'-(V')_{\cD_2}|\big)_{\cD_2}
+ N \big(|\hat \Theta-(\hat \Theta)_{\cD_2}|\big)_{{\cD_2}}
+ N \big(W\big)_{\cD_2}\nonumber\\
&\le  N\kappa^{-1/2}(U^2)_{B^+_{R/2}(\hat y_0)}^{1/2}
+N\kappa^{\frac d 2} \left(\gamma^{1/ {\nu'}} (U^\nu)_{\Omega_{R}(\hat y_0)}^{1/ \nu}+
(F^2)_{\Omega_{R}(\hat y_0)}^{1/2}\right).
\end{align*}
By using the triangle inequality and the assumption $\kappa \gamma\le 1/10$,
$$
\big(|U^B-(U^B)_{\Omega_r(y_0)}|\big)_{\Omega_r(y_0)}
\le N\big(|U^B-(U^B)_{\cD_2}|\big)_{\cD_2}
$$
$$
+N\kappa\gamma
(U^B)_{\cD_2}
+N(1_{\cD_1} U^B )_{\Omega_r(y_0)}.
$$
We use \eqref{eq12.01}, \eqref{eq21.31}, and H\"older's inequality to bound the last two terms on the right-hand side above as follows:
$$
\kappa \gamma (U^B)_{\cD_2}\le N\kappa \gamma (|U|^2)^{1/2}_{\cD_2}
\le N\kappa \gamma \kappa^{d/2}(|U|^2)^{1/2}_{\Omega_R(\hat y_0)},
$$
\begin{equation*}
(1_{\cD_1} U^B )_{\Omega_r(y_0)}\le (1_{\cD_1})^{1/2}_{\Omega_r(y_0)}(|U|^2)^{1/2}_{\Omega_r(y_0)}
\le N(\kappa \gamma)^{1/2} \kappa^{d/2}(|U|^2)_{\Omega_R(\hat y_0)}^{1/2}.
\end{equation*}
Therefore,
\begin{multline}
                                \label{eq21.12}
\big(|U^B-(U^B)_{\Omega_r(y_0)}|\big)_{\Omega_r(y_0)}
\le  N\big(\kappa^{-1/2}+(\kappa \gamma)^{1/2} \kappa^{d/2}\big) (|U|^2)_{\Omega_R(\hat y_0)}^{1/2}\\
+N\kappa^{\frac d 2} \left(\gamma^{1/ {\nu'}} (U^\nu)_{\Omega_{R}(\hat y_0)}^{1/ \nu}+(F^2)_{\Omega_{R}(\hat y_0)}^{1/2}\right).
\end{multline}

In the remaining case when $\kappa\gamma>1/10$, by \eqref{eq12.01} and \eqref{eq22.16},
\begin{align*}
\big(|U^B-(U^B)_{\Omega_r(y_0)}|\big)_{\Omega_r(y_0)}&\le N(U^B)_{\Omega_r(y_0)}\le N(U)_{\Omega_r(y_0)}\\
&\le N(|U|^2)_{\Omega_r(y_0)}^{1/2}
\le N\kappa^{d/2}(|U|^2)_{\Omega_R(\hat y_0)}^{1/2},
\end{align*}
where in the last inequality, we used the obvious inequality $|\Omega_R(\hat y_0)|\le N\kappa^d|\Omega_r(y_0)|$.
Therefore, in this case, \eqref{eq21.12} still holds.
Finally, we transform the obtained inequality back to the original coordinates to get the inequality \eqref{eq5.42}.
This completes the proof of Proposition \ref{prop7.9}.

\subsection{Proof of Theorem \ref{thm3}}
We finish the proof of Theorem \ref{thm3} in this subsection.
First we observe that by taking a sufficiently large $\lambda_0$ and using interpolation inequalities, we can move all the lower-order terms of $\cL \vu$ to the right-hand side.
Thus in the sequel we assume that all the lower-order coefficients of $\cL$ are zero.

Recall the definition of $\bC_l,l\in \bZ$ above Theorem \ref{th081201}.
Notice that if $x\in C\in \bC_l$, then for the smallest $r>0$
such that $C\subset B_r(x)$ we have
$$
\dashint_{C} \dashint_{C}|g(y)-g(z)|
\,dy\,dz\leq N(d)
\dashint_{B^+_r(x)} \dashint_{B^+_r(x)}|g(y)-g(z)|
\,dy\,dz.
$$
We use this inequality in the proof of the following corollary.

\begin{corollary}                           \label{cor001b}
Let $\gamma \in (0,1/4)$, $\lambda > 0$, $\nu\in (2,\infty)$, $\nu'=2\nu/(\nu-2)$, and
$z_0 \in \overline{\bR^d_+}$.
Assume that $\vu\in W_{\nu,\text{loc}}^m(\overline{\bR^d_+})$ vanishes outside $B_{\gamma R_0}(z_0)$ and
satisfies
$$
(-1)^m\cL \vu+\lambda \vu=\sum_{|\alpha|\le m}D^\alpha \vf_\alpha
$$
locally in $\bR^d_+$ with the conormal derivative condition on $\partial \bR^d_+$,
where $\vf_\alpha\in L_{2,\text{loc}}(\overline{\bR^d_+})$.
Then under Assumption \ref{assumption20100901} ($\gamma$),
for each $l \in \bZ$, $C \in \bC_l$, and $\kappa \ge 64$,
there exists a function $U^C$ depending on $C$
such that $N^{-1}U\le U^C\le NU$ and
\begin{equation}
							\label{eq08_04}
\left(|U^C-(U^C)_{C}|\right)_{C}
\le N \left(F_{\kappa}\right)_C,
\end{equation}
where
$N=N(d,\delta,m,n,\tau)$ and
\begin{align*}
F_{\kappa}&=
\big(\kappa^{-1/2}+(\kappa \gamma)^{1/2}\kappa^{d/2}\big)\big(\cM (1_{\bR^d_+}U^2)\big)^{1/2}\\
&\quad +\kappa^{d/2}\big[\big(\cM(1_{\bR^d_+}F^2)\big)^{1/2}+\gamma^{1/\nu'}
(\cM(1_{\bR^d_+}U^{\nu}))^{1/\nu}\big].
\end{align*}
\end{corollary}

\begin{proof}
For each $\kappa \ge 64$ and $C \in \bC_l$, let $B_r(x_0)$ be the smallest ball containing $C$.
Clearly, $x_0 \in \overline{\bR^d_+}$.

If $\kappa r > R_0$, then we take $U^C = U$.
Note that the volumes of $C$, $B_r(x_0)$, and $B^+_r(x_0)$ are comparable, and $C\subset B_r^+(x_0)$.
Then by the triangle inequality and H\"{o}lder's inequality, the left-hand side of \eqref{eq08_04} is less than
\begin{multline*}
N \left(U^2\right)^{1/2}_{B_r^+(x_0)}
\le N \kappa^{d/2} \big( 1_{B^+_{\gamma R_0}(z_0)}U^2 \big)^{1/2}_{B_{\kappa r}(x_0)}\\
\le N \kappa^{d/2} \big( 1_{B_{\gamma R_0(z_0)}} \big)_{B_{\kappa r}(x_0)}^{1/\nu'} \big( 1_{\bR^d_+ }U^\nu \big)^{1/\nu}_{B_{\kappa r}(x_0)}
\le N \kappa^{d/2} \gamma^{1/\nu'} \big( 1_{\bR^d_+ }U^\nu \big)^{1/\nu}_{B_{\kappa r}(x_0)}.
\end{multline*}
Here the first inequality is because $|B_{\kappa r}(x_0)|\le 2\kappa^d |B_r^+(x_0)|$ and $U$ vanishes outside $B_{\gamma R_0}(z_0)$. The second inequality follows from H\"older's inequality. In the last inequality we used $\kappa r>R_0$ and $\gamma^d \le \gamma$.
Note that
\begin{equation}
							\label{eq08_05}
\big( 1_{\bR^d_+ }U^\nu \big)^{1/\nu}_{B_{\kappa r}(x_0)}
\le
\cM^{1/\nu} \big(1_{\bR^d_+ }U^\nu \big)(x)
\end{equation}
for all $x \in C$.
Hence the inequality \eqref{eq08_04} follows.

If $\kappa r \le R_0$, from Proposition \ref{prop7.9}, we find $U^B$ with $B^+ = B^+_{\kappa r}(x_0)$.
Take $U^C = U^B$. Then by Proposition \ref{prop7.9} we have
\begin{equation}
							\label{eq08_06}
\left( |U^C - \left(U^C\right)_C| \right)_C \le N(d) I,
\end{equation}
where $I$ is the right-hand side of the inequality \eqref{eq5.42}. Note that, for example,
$$
\left( U^2 \right)_{B_{\kappa r}^+(x_0)}
\le N(d) \big( 1_{\bR^d_+}U^2 \big)_{B_{\kappa r}(x_0)}.
$$
Using this and inequalities like \eqref{eq08_05}, we see that \eqref{eq08_06} implies the desired inequality \eqref{eq08_04}.
\end{proof}

\begin{theorem}                         \label{theorem001b}
Let $p \in (2,\infty)$, $\lambda > 0$, $z_0\in \overline{\bR^d_+}$,
and $\vf_{\alpha} \in L_p(\bR^d_+)$.
There exist positive constants $\gamma\in (0,1/4)$ and $N$,
depending only on $d$, $\delta$, $m$, $n$, $p$, such that under Assumption \ref{assumption20100901} ($\gamma$),
for $\vu \in W_p^m(\bR^d_+)$ vanishing outside $B_{\gamma R_0}(z_0)$ and satisfying
$$
(-1)^m\cL \vu+\lambda \vu=\sum_{|\alpha|\le m}D^\alpha \vf_\alpha
$$
in $\bR^d_+$ with the conormal derivative condition on $\partial \bR^d_+$,
we have
$$
\|U\|_{L_p(\bR^d_+)}
\le N \| F\|_{L_p(\bR^d_+)},
$$
where $N = N(d,\delta,m,n,p)$.
\end{theorem}

\begin{proof}
Let $\gamma > 0$ and $\kappa \ge 64$ be constants to be specified below.
Take a  constant $\nu$ such that $p > \nu > 2$.
Then we see that $\vu \in W_{\nu,\text{loc}}^m(\overline{\bR^d_+})$ and all the conditions in Corollary \ref{cor001b} are satisfied.

For each $l \in \bZ$ and $C \in \bC_l$, let $U^C$ be the function from Corollary \ref{cor001b}.
Then by Corollary \ref{cor001b} and Theorem \ref{th081201}
we have
$$
\| U \|_{L_p(\bR^d_+)}^p \le N \|F_\kappa\|_{L_p(\bR^d_+)} \|U\|_{L_p(\bR^d_+)}^{p-1}.
$$
The implies that
$$
\|U \|_{L_p(\bR^d_+)} \le N \|F_\kappa\|_{L_p(\bR^d_+)}.
$$
Now we observe that by the Hardy--Littlewood maximal function theorem
\begin{multline*}
\|F_\kappa\|_{L_p(\bR^d_+)}
\le \|F_\kappa\|_{L_p(\bR^d)}
\le N \big(\kappa^{-1/2} + (\kappa \gamma)^{1/2} \kappa^{d/2}\big)\|1_{\bR^d_+} U\|_{L_p(\bR^d)}\\
+ N \kappa^{d/2} \|1_{\bR^d_+} F\|_{L_p(\bR^d)} + N \kappa^{d/2}\gamma^{1/\nu'} \|1_{\bR^d_+} U\|_{L_p(\bR^d)}.
\end{multline*}
To complete the proof, it remains to choose a sufficiently large $\kappa$, and then a sufficiently small $\gamma$
so that
$$
N \big(\kappa^{-1/2} + (\kappa \gamma)^{1/2} \kappa^{d/2}\big)
+ N \kappa^{d/2}\gamma^{1/\nu'} < 1/2.
$$
\end{proof}

\begin{proof}[Proof of Theorem \ref{thm3}]
We treat the following three cases separately.

{\em Case 1: $p=2$.} In this case, the theorem follows from Theorem \ref{theorem08061901}.

{\em Case 2: $p\in (2,\infty)$.} Assertion (i) follows from Theorem \ref{theorem001b} and the standard partition of unity argument. Then Assertion (ii) is derived from Assertion (i) by using the method of continuity. Finally, Assertion (iii) is due to a standard scaling argument.

{\em Case 3: $p\in (1,2)$.} In this case, Assertion (i) is a consequence of the duality argument and the $W^m_q$-solvability obtained above for $q=p/(p-1)\in (2,\infty)$. With the a priori estimate, the remaining part of the theorem is proved in the same way as in Case 2. The theorem is proved.
\end{proof}

\section{Systems on a Reifenberg flat domain}
                            \label{Reifenberg}

In this section, we consider elliptic systems on a Reifenberg flat domain. The crucial ingredients of the proofs below are the interior and the boundary estimates established in Sections \ref{sec_aux}, a result in \cite{Sa80,KS80} on the ``crawling of ink drops'', and an idea in \cite{CaPe98}.

By a scaling, in the sequel we may assume $R_0=1$ in Assumption \ref{assump1}. Recall the definitions of $U$ and $F$ in Sections \ref{sec_aux} and \ref{sec4}.

\begin{lemma}
                            \label{lem7.3}
Let $\gamma \in (0,1/50)$, $R\in (0,1]$, $\lambda\in (0,\infty)$, $\nu\in (2,\infty)$, $\nu'=2\nu/(\nu-2)$, $\vf_\alpha= (f_\alpha^1, \ldots, f_\alpha^n)^{\text{tr}} \in L_{2,\text{loc}}(\overline{\Omega})$, $|\alpha|\le m$. Assume that $a_{\alpha\beta}\equiv 0$ for any $\alpha,\beta$ satisfying $|\alpha|+|\beta|<2m$ and that $\vu\in W_{\nu, \text{loc}}^m(\overline{\Omega})$ satisfies \eqref{eq11_01}
locally in $\Omega$ with the conormal derivative condition on $\partial \Omega$.
Then the following hold true.

\noindent(i) Suppose $0\in \Omega$, $\dist(0,\partial\Omega)\ge R$, and Assumption \ref{assump1} ($\gamma$) (i) holds at the origin. Then there exists nonnegative functions $V$ and $W$ in $B_R$ such that $U \le V+W$ in $B_R$, and $V$ and $W$ satisfy
\begin{equation*}
(W^2)_{B_R}^{1/2}
\le N\gamma^{1/\nu'} (U^\nu)_{B_R}^{1/\nu}+
N(F^2)_{B_R}^{1/2}
\end{equation*}
and
\begin{equation*}
\|V\|_{L_\infty(B_{R/4})}
\le N\gamma^{1/\nu'} (U^\nu)_{B_{R}}^{1/ \nu}+
N(F^2)_{B_{R}}^{1/2}+
N(U^2)_{B_{R}}^{1/2},
\end{equation*}
where $N=N(d,n,m,\delta,\nu)>0$ is a constant.

\noindent(ii) Suppose $0\in \partial\Omega$ and Assumption \ref{assump1} ($\gamma$) (ii) holds at the origin. Then there exists nonnegative functions $V$ and $W$ in $\Omega_R$ such that $U \le V+W$ in $\Omega_R$, and $W$ and $V$ satisfy
\begin{equation}
            \label{eq17.10}
(W^2)_{\Omega_R}^{1/2}
\le N\gamma^{1 /\nu'} (U^\nu)_{\Omega_R}^{1/\nu}+
N(F^2)_{\Omega_R}^{1/2}
\end{equation}
and
\begin{equation} \label{eq17.11}
\|V\|_{L_\infty(\Omega_{R/4})} \le N\gamma^{1/\nu'}
(U^\nu)_{\Omega_R}^{1/\nu}
+N(F^2)_{\Omega_R}^{1/2}+ N(U^2)_{\Omega_R}^{1/2},
\end{equation}
where $N=N(d,n,m,\delta,\nu)>0$ is a constant.
\end{lemma}

\begin{proof}
The proof is similar to that of Proposition \ref{prop7.9} with some modifications.
We assume that Assumption \ref{assump1} holds in the original coordinates.
Without loss of generality, we may further  assume that the coefficients $\bar a_{\alpha\beta}$ are infinitely differentiable.

Assertion (i) is basically an interior estimate which does not involve boundary conditions, so the proof is exactly the same as that of Assertion (i) in \cite[Lemma 8.3]{DK10}.

Next, we prove Assertion (ii).
Due to Assumption \ref{assump1}, by shifting the origin upward, we can assume that
$$
B_R^+(x_0)
\subset \Omega_R(x_0) \subset \{(x_1,x') : -2\gamma R < x_1 \} \cap B_R(x_0)
$$
where $x_0 \in \partial \Omega$ (see Figure \ref{fg2}). Define $\bar a_{\alpha\beta}$ as in Section \ref{sec4}.
\begin{figure}
\centering
\includegraphics{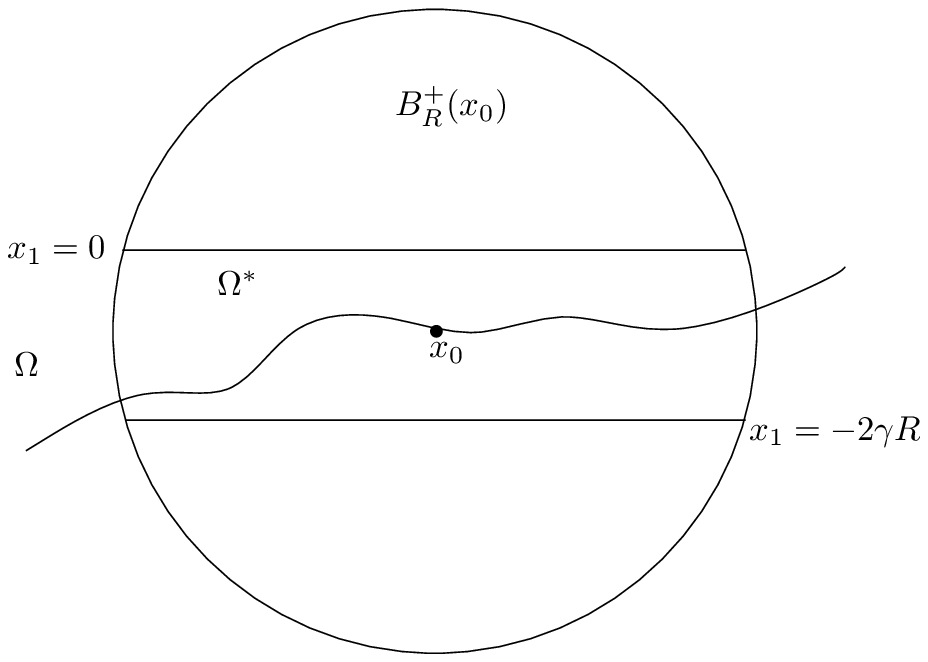}
\caption{}\label{fg2}
\end{figure}
Then $\vu$ satisfies \eqref{eq17.23b} in $B_R^+(x_0)$ with the conormal derivative condition on $\Gamma_R(x_0)$.
By following the argument in the proof of Proposition \ref{prop7.9}, we can find $\vw \in W_2^m(\bR^d_+)$ and $\vv \in W_{2}^m(B_R^+(x_0))$ such that $\vu = \vw + \vv$, the function $\vw$ satisfies
\begin{equation}
							\label{eq08_07}
\sum_{k=0}^m\lambda^{\frac 1 2-\frac k {2m}}(|D^k  \vw|^2)_{B_R^+(x_0)}^{1/2}
\le N\gamma^{1/ {\nu'}} (U^\nu)_{\Omega_R(x_0)}^{1/ \nu}+
N(F^2)_{\Omega_R(x_0)}^{1/2},
\end{equation}
and the function $\vv$ satisfies
$$
(-1)^m \cL \vv + \lambda \vv = 0
$$
in $B_{R/2}^+(x_0)$ with the conormal derivative condition on $\Gamma_{R/2}(x_0)$. We define $V$ and $W$ in $B_R^+(x_0)$ as in Section \ref{sec4}.
As noted in the proof of Proposition \ref{prop7.9}, we can assume that $\vv$ is infinitely differentiable.
Applying Lemma \ref{lem6.7}, we get
\begin{equation}
							\label{eq08_08}
\|V\|_{L_\infty(B^+_{R/4}(x_0))}
\le N (V^2)_{B^+_{R/2}(x_0)}^{1/2}.
\end{equation}
Now we extend $W$ and $V$ on $\Omega^* = \bR^d_- \cap \Omega_R(x_0)$
by setting $W = U$ and $V = 0$, respectively.
Then we see that by H\"{o}lder's inequality and \eqref{eq08_07}
\begin{align*}
\left(W^2\right)^{1/2}_{\Omega_R(x_0)}
&= \left[ \frac{1}{|\Omega_R(x_0)|}\int_{B_R^+(x_0)} W^2 \, dx
+ \frac{1}{|\Omega_R(x_0)|} \int_{\Omega^*} U^2 \, dx \right]^{1/2}\\
&\le N \left( W^2 \right)_{B_R^+(x_0)}^{1/2}
+ N \left( 1_{\Omega^*} \right)_{\Omega_R(x_0)}^{1/\nu'}
\left( U^\nu \right)^{1/\nu}_{\Omega_R(x_0)}\\
&\le N\gamma^{1/ {\nu'}} (U^\nu)_{\Omega_R(x_0)}^{1/ \nu}+
N(F^2)_{\Omega_R(x_0)}^{1/2}.
\end{align*}
Upon recalling that the origin was shifted from $x_0$, we arrive at \eqref{eq17.10}.
To prove \eqref{eq17.11} we observe that by \eqref{eq08_08} and the fact that $V\le U+W$
$$
\|V\|_{L_\infty(\Omega_{R/4}(x_0))}
= \|V\|_{L_\infty(B_{R/4}^+(x_0))}
\le N \left(V^2\right)^{1/2}_{B_{R/2}^+(x_0)}
$$
$$
\le N \left(V^2\right)^{1/2}_{\Omega_R(x_0)}
\le N \left(W^2\right)^{1/2}_{\Omega_R(x_0)}
+ N \left(U^2\right)^{1/2}_{\Omega_R(x_0)}.
$$
This together with \eqref{eq17.10} gives \eqref{eq17.11}.
This completes the proof of the lemma.
\end{proof}

For a function $f$ on a set $\cD\subset\bR^{d}$, we define
its maximal function $\cM f$ by $\cM f=\cM (I_{\cD}f)$.
For any $s>0$, we introduce two level sets
$$
\cA(s)=\{x\in \Omega:U> s\},
$$
$$
\cB(s)=\Big\{x\in \Omega:\gamma^{-1/\nu'}\big(\cM(F^2)\big)^{1/2}+
\big(\cM(U^\nu)\big)^{1/\nu}>s\Big\}.
$$
With Lemma \ref{lem7.3} in hand, we get the following corollary.

\begin{corollary}
                                    \label{cor7.5}
Under the assumptions of Lemma \ref{lem7.3}, suppose $0\in \overline{\Omega}$ and Assumption \ref{assump1} ($\gamma$) holds. Let $s\in (0,\infty)$ be a constant. Then there exists a constant $\kappa\in (1,\infty)$, depending only on $d$, $n$, $m$, $\delta$, and $\nu$, such that the following holds.
If
\begin{equation}
                    \label{eq15.59}
\big|\Omega_{R/32}\cap \cA(\kappa s)\big|> \gamma^{2/\nu'} |\Omega_{R/32}|,
\end{equation}
then we have $\Omega_{R/32}\subset \cB(s)$.
\end{corollary}
\begin{proof}
By dividing $\vu$ and $\vf$ by $s$, we may assume $s=1$.
We prove by contradiction. Suppose at a point $x\in \Omega_{R/32}$, we have
\begin{equation}
                            \label{eq15.38}
\gamma^{-1/\nu'}\big(\cM(F^2)(x)\big)^{1/2}+
\big(\cM(U^\nu)(x)\big)^{1/\nu}\le 1.
\end{equation}
Let us consider two cases.

First we consider the case when $\dist(0,\partial \Omega)\ge R/8$.
Notice that
$$
x\in\Omega_{R/32}=B_{R/32}\subset B_{R/8}\subset  \Omega.
$$
Due to Lemma \ref{lem7.3} (i), we have $U \le V + W$ and, by \eqref{eq15.38},
\begin{equation}
                                \label{eq15.47}
\|V\|_{L_\infty(B_{R/32})}
\le N_1,
\quad
(W^2)_{B_{R/8}}^{1/2}
\le N_1\gamma^{1/\nu'},
\end{equation}
where $N_1$ and constants $N_i$ below depend only on $d$, $n$, $m$, $\delta$, and $\nu$. By \eqref{eq15.47}, the triangle inequality and Chebyshev's inequality, we get
\begin{multline}
                                \label{eq5.20}
\big|\Omega_{R/32}\cap \cA(\kappa)\big|=
\big|\{x\in \Omega_{R/32}: U>\kappa\}\big|\\
\le \big|\{x\in \Omega_{R/32}: W>\kappa-N_1\}\big|
\le (\kappa-N_1)^{-2} N_1^2 \gamma^{2/\nu'}|B_{R/8}|,
\end{multline}
which contradicts with \eqref{eq15.59} if we choose $\kappa$ sufficiently large.

Next we consider the case when $\dist(0,\partial \Omega)< R/8$. We take $y\in \partial\Omega$ such that $|y|=\dist(0,\partial \Omega)$. Notice that in this case we have
$$
x\in\Omega_{R/32}\subset \Omega_{R/4}(y)\subset \Omega_{R}(y).
$$
Due to Lemma \ref{lem7.3} (ii), we have $U\le V+ W$ in $\Omega_R(y)$ and, by \eqref{eq15.38},
\begin{equation}
                                \label{eq17.23c}
\|V\|_{L_\infty(\Omega_{R/32})}\le \|V\|_{L_\infty(\Omega_{R/4}(y))}
\le N_2,
\quad
(W^2)_{\Omega_{R}(y)}^{1/2}
\le N_2\gamma^{1/\nu'}.
\end{equation}
By \eqref{eq17.23c}, the triangle inequality and Chebyshev's inequality, we still get \eqref{eq5.20} with $N_2$ in place of $N_1$,
which contradicts with \eqref{eq15.59} if we choose $\kappa$ sufficiently large.
\end{proof}

\begin{theorem}                         \label{theorem101}
Let $p \in (2,\infty)$, $\lambda > 0$, $x_0\in \bR^{d}$
and $\vf_{\alpha} \in L_p(\Omega)$. Suppose that $a_{\alpha\beta}\equiv 0$ for any $\alpha,\beta$ satisfying $|\alpha|+|\beta|<2m$, and $\vu \in W_p^m(\Omega)$ is supported on $B_{\gamma}(x_0)\cap \overline{\Omega}$ and satisfies \eqref{eq11_01} in $\Omega$ with the conormal derivative boundary condition.
There exist positive constants $\gamma \in (0,1/50)$ and $N$,
depending only on $d,\delta,m,n, p$, such that, under Assumption \ref{assump1} ($\gamma$)
we have
\begin{equation}
                                \label{eq16.00}
\|U \|_{L_p(\Omega)}
\le N \|F\|_{L_p(\Omega)},
\end{equation}
where $N = N(d,\delta,m,n,p)$.
\end{theorem}

\begin{proof}
We fix $\nu=p/2+1$ and let $\nu'=2\nu/(\nu-2)$. Then we see that $\vu \in W_{\nu,\text{loc}}^m(\overline{\Omega})$. Let $\kappa$ be the constant in Corollary \ref{cor7.5}. Recall the elementary identity:
$$
\|f\|_{L_p(\cD)}^p=p\int_0^\infty \big|\{x\in\cD:|f(x)|> s\}\big|s^{p-1}\,ds,
$$
which implies that
\begin{equation}
                                        \label{eq5.07}
\|U\|_{L_p(\Omega)}^p=p\kappa^p \int_0^\infty|\cA(\kappa s)|s^{p-1}\,ds.
\end{equation}
Thus, to obtain \eqref{eq16.00} we need to estimate $\cA(\kappa s)$.
First, we note that by Chebyshev's inequality
\begin{equation}
                            \label{eq17.57}
|\cA(\kappa s)|\le (\kappa s)^{-2}\|U\|_{L_2(\Omega)}^2.
\end{equation}
When $\kappa s \ge \gamma^{-1/\nu'} \|U\|_{L_2(\Omega)}$, this indicates that
$$
|\cA(\kappa s)|\le \gamma^{2/\nu'}.
$$
With the above inequality and Corollary 5.2 in hand, we see that all the conditions of the ``crawling of ink drops'' lemma are satisfied; see \cite{Sa80}, \cite[Sect. 2]{KS80}, or \cite[Lemma 3]{BW05} for the lemma. Hence we have
\begin{equation}
                            \label{eq18.34}
|\cA(\kappa s)|\le  N_4\gamma^{2/\nu'}|\cB(s)|.
\end{equation}
Now we estimate $\|U\|_{L_p(\Omega)}^p$ in \eqref{eq5.07} by splitting the integral into two depending on the range of $s$. If $\kappa s  \ge \gamma^{-1/\nu'}\|U\|_{L_2(\Omega)}$, we use \eqref{eq18.34}. Otherwise, we use  \eqref{eq17.57}. Then it follows that
\begin{align*}
\|U\|_{L_p(\Omega)}^p &\le N_5\gamma^{(2-p)/\nu'}\big(\|U\|_{L_2(\Omega)}^p+
\big\|\big(\cM(F^2)\big)^{1/2}\big\|_{L_p(\Omega)}^p\big)\\
&\quad +N_5\gamma^{2/\nu'}\big\|\big(\cM(U^\nu)\big)^{1/\nu}
\big\|_{L_p(\Omega)}^p\\
&\le N_5\gamma^{(2-p)/\nu'}\|U\|_{L_2(\Omega)}^p+N_6\gamma^{(2-p)/\nu'}
\|F\|_{L_p(\Omega)}^p+N_6\gamma^{2/\nu'}\|U\|_{L_p(\Omega)}^p,
\end{align*}
where we used the Hardy--Littlewood maximal function theorem in the last inequality.
By H\"older's inequality,
\begin{equation}
                            \label{eq20.31bb}
\|U\|_{L_2(\Omega)}=\|U\|_{L_2(B_\gamma(x_0)\cap\Omega)}
\le N\|U\|_{L_p(\Omega)}\gamma^{d(1/2-1/p)}.
\end{equation}
By the choice of $\nu$, $d(p/2-1)+(2-p)/\nu'>2/\nu'$.
Thus, we get
$$
\|U\|_{L_p(\Omega)}^p \le N_6\gamma^{(2-p)/\nu'}
\|F\|_{L_p(\Omega)}^p
+N_6\gamma^{2/\nu'}\|U\|_{L_p(\Omega)}^p.
$$
To get the estimate \eqref{eq16.00}, it suffices to take $\gamma=\gamma(d,n,m,\delta,p)\in (0,1/50]$ sufficiently small such that $N_6\gamma^{2/\nu'}\le 1/2$.
\end{proof}

\begin{proof}[Proof of Theorem \ref{thm5}]
We again consider the following three cases separately.

{\em Case 1: $p=2$.} In this case, the theorem follows directly from the well-known Lax--Milgram lemma.

{\em Case 2: $p\in (2,\infty)$.} Assertion (i) follows from Theorem \ref{theorem101} and the standard partition of unity argument. By the method of continuity, for Assertion (ii) it suffices to prove the solvability for the operator $\cL_1:=\delta_{ij}\Delta^m$, which is not immediate because the domain $\Omega$ is irregular. We approximate $\Omega$ by regular domains. Recall the definition of $\Omega^\varepsilon$ in \eqref{eq2.22}. By Lemma \ref{lem3.11}, for any $\varepsilon\in (0,1/4)$, $\Omega^\varepsilon$ satisfies Assumption \ref{assump3} ($N_0 \gamma^{1/2}$) with a constant $R_1(\varepsilon)>0$. Thanks to Lemma \ref{lem3.19}, there is a sequence of expanding smooth domains $\Omega^{\varepsilon,k}$ which converges to $\Omega^\varepsilon$ as $k\to \infty$. Moreover, $\Omega^{\varepsilon,k}$ satisfies Assumption \ref{assump3} ($N_0 \gamma^{1/2}$) with the constant $R_1(\varepsilon)/2$ which is independent of $k$. In particular, $\Omega^{\varepsilon,k}$ satisfies Assumption \ref{assump1} ($N_0\gamma^{1/2}$) with the constant $R_1(\varepsilon)/2$. By the classical result, there is a constant $\lambda_\varepsilon=\lambda_\varepsilon(d,n,m,p,\delta)\ge \lambda_0$ such that, for any $\lambda>\lambda_\varepsilon$, the equation
$$
(-1)^m\cL_1\vu+\lambda \vu=\sum_{|\alpha|\le m}D^\alpha \vf_\alpha\quad \text{in}\,\,\Omega^{\varepsilon,k}
$$
with the conformal derivative boundary condition has a unique solution $\vu^{\varepsilon,k}\in W^{m}_p(\Omega^{\varepsilon,k})$. The a priori estimate above gives
\begin{equation}
                                        \label{eq3.07}
\|\vu^{\varepsilon,k}\|_{W^m_p(\Omega^{\varepsilon,k})}\le N_\varepsilon,
\end{equation}
where $N_\varepsilon>0$ is a constant independent of $k$. By the weak compactness, there is a subsequence, which is still denoted by $\vu^{\varepsilon,k}$, and functions $\vv^\varepsilon,\vv^\varepsilon_\alpha\in L_p(\Omega^{\varepsilon}),1\le |\alpha|\le m$, such that weakly in $L_p(\Omega^{\varepsilon})$,
$$
\vu^{\varepsilon,k}I_{\Omega^{\varepsilon,k}}\rightharpoonup \vv^\varepsilon,\quad
D^\alpha \vu^{\varepsilon,k}I_{\Omega^{\varepsilon,k}}\rightharpoonup \vv^\varepsilon_\alpha\quad \forall\, \alpha,\,1\le |\alpha|\le m.
$$
It is easily seen that  $\vv^\varepsilon_\alpha=D^\alpha \vv^\varepsilon$ in the sense of distributions. Thus, by \eqref{eq3.07} and the weak convergence,
$\vv^{\varepsilon}\in W^m_p(\Omega^\varepsilon)$ is a solution to
\begin{equation}
                                        \label{eq4.58}
(-1)^m\cL_1\vu+\lambda \vu=\sum_{|\alpha|\le m}D^\alpha \vf_\alpha\quad \text{in}\,\,\Omega^{\varepsilon}
\end{equation}
with the conormal derivative boundary condition. We have proved the solvability for any $\lambda>\lambda_\varepsilon$. Recall that, by Lemma \ref{lem3.11}, $\Omega^\varepsilon$ satisfies Assumption \ref{assump1} ($N_0\gamma^{1/2}$) with $R_0=1/2$. By the a priori estimate in Assertion (i) and the method of continuity, for any $\lambda>\lambda_0$ there is a unique solution $\vu^\varepsilon\in W^m_p(\Omega^\varepsilon)$ to \eqref{eq4.58} with the conormal derivative boundary condition. Moreover, we have
\begin{equation}
                                        \label{eq3.07b}
\|\vu^{\varepsilon}\|_{W^m_p(\Omega^{\varepsilon})}\le N,
\end{equation}
where $N$ is a constant independent of $\varepsilon$. Again
by the weak compactness, there is a subsequence $\vu^{\varepsilon_j}$, and functions $\vu,\vu_\alpha\in L_p(\Omega),1\le |\alpha|\le m$, such that weakly in $L_p(\Omega)$,
$$
\vu^{\varepsilon_j}I_{\Omega^{\varepsilon_j}}\rightharpoonup \vu,\quad
D^\alpha \vu^{\varepsilon_j}I_{\Omega^{\varepsilon_j}}\rightharpoonup \vu_\alpha\quad \forall\, \alpha,\,1\le |\alpha|\le m.
$$
It is easily seen that  $\vu_\alpha=D^\alpha \vu$ in the sense of distributions. Thus, by \eqref{eq3.07b} and the weak convergence,
$\vu\in W^m_p(\Omega)$ is a solution to
\begin{equation*}
(-1)^m\cL_1\vu+\lambda \vu=\sum_{|\alpha|\le m}D^\alpha \vf_\alpha\quad \text{in}\,\,\Omega
\end{equation*}
with the conormal derivative boundary condition. The uniqueness then follows from the a priori estimate. This completes the proof of Assertion (ii).

{\em Case 3: $p\in (1,2)$.} The a priori estimate in Assertion (i) is a directly consequence of the solvability when $p\in (2,\infty)$ and the duality argument. Then the solvability in Assertion (ii) follows from the a priori estimate by using the same argument as in Case 2.

The theorem is proved.
\end{proof}


We now give the proofs of Corollary \ref{cor7} and Theorem \ref{thmB}.

\begin{proof}[Proof of Corollary \ref{cor7}]
{\em Case 1: $p=2$.} We define a Hilbert space
$$
H:=\{\vu\in W^m_2(\Omega)\,|\,(\vu)_\Omega=(D\vu)_\Omega=\ldots=(D^{m-1}\vu)_\Omega=0\}.
$$
By the Lax--Milgram lemma, there is a unique $\vu\in H$ such that for any $\vv\in H$,
\begin{equation}
                        \label{eq22.44}
\int_\Omega a_{\alpha\beta}D^\beta \vu D^\alpha \vv\,dx=
\sum_{|\alpha|=m}\int_\Omega (-1)^{|\alpha|}\vf_\alpha D^\alpha \vv
\end{equation}
and
$$
\|D^m \vu\|_{L_2(\Omega)}\le N\sum_{|\alpha|=m}\|\vf_\alpha\|_{L_2(\Omega)}.
$$
Note that any function $\vv\in W^m_2(\Omega)$ can be decomposed as a sum of a function in $H$ and a polynomial of degree at most $m-1$. Therefore, \eqref{eq22.44} also holds for any $\vv\in W^m_2(\Omega)$. This implies that $\vu\in W_2^m(\Omega)$ is a solution to the original equation. On the other hand, by the uniqueness of the solution in $H$, any solution $\vw\in W_2^m(\Omega)$ can only differ from $\vu$ by a polynomial of order at most $m-1$.

{\em Case 2: $p\in (2,\infty)$.} First we suppose that $p$ satisfies $1/p>1/2-1/d$. Since $\Omega$ is bounded, $\vf\in L_2(\Omega)$. Let $\vu$ be the unique $H$-solution to the equation. By Theorem \ref{thm5}, there is a unique solution $\vv\in W^m_p(\Omega)$ to the equation
\begin{equation}
                                        \label{eq6.01}
(-1)^m\cL \vv+(\lambda_0+1)\vv=\sum_{|\alpha|= m}D^\alpha \vf_\alpha+(\lambda_0+1)\vu\quad \text{in}\,\,\Omega
\end{equation}
with the conormal derivative boundary condition.
Moreover, we have
\begin{equation}
                                \label{eq23.17}
\|\vv \|_{W^{m}_p(\Omega)}
\le N \sum_{|\alpha|= m}\|\vf_\alpha \|_{L_p(\Omega)}+N\|\vu \|_{L_p(\Omega)}.
\end{equation}
By the Sobolev imbedding theorem and the $W^m_2$ estimate, we have
$$
\|\vu\|_{L_p(\Omega)}\le N\|\vu\|_{W^1_2(\Omega)}\le N\sum_{|\alpha|= m}\|\vf_\alpha\|_{L_2(\Omega)}\le N\sum_{|\alpha|= m}\|\vf_\alpha\|_{L_{p}(\Omega)},
$$
which together with \eqref{eq23.17} gives
$$
\|\vv \|_{W^{m}_p(\Omega)}\le N \sum_{|\alpha|= m}\|\vf_\alpha \|_{L_p(\Omega)}.
$$
Since both $\vv$ and $\vu$ are $W^m_2(\Omega)$-solutions to \eqref{eq6.01}, by Theorem \ref{thm5} we have $\vu=\vv$. Therefore, the solvability is proved under the assumption $1/p>1/2-1/d$. The general case follows by using a bootstrap argument. The uniqueness is due to the uniqueness of $W^m_2$-solutions.

{\em Case 3: $p\in (1,2)$.} By the duality argument and Case 2, we have the a priori estimate \eqref{eq23.08} for any $\vu\in W^m_p(\Omega)$ satisfying \eqref{eq22.34} with the conormal derivative boundary condition. For the solvability, we take a sequence
$$
\vf^{\,k}_\alpha=\min\{\max\{\vf_\alpha,-k\},k\}\in L_2(\Omega)
$$
which converges to
$\vf_\alpha$ in $L_p(\Omega)$. Let $\vu^k$ be the
$H$-solution to the equation with the right-hand side $\vf^{\,k}_\alpha$.
Since $\Omega$ is bounded, we have $\vu^k\in W^m_p(\Omega)$.
By the a priori estimate, $\vu^k$ is a Cauchy sequence
in $W^m_p(\Omega)$. Then it is easily seen that the limit $\vu$ is the $W^m_p(\Omega)$-solution to the original equation. Next we show the uniqueness. Let $\vu_1$ be another $W^m_p(\Omega)$-solution to the equation. Then $\vv:=\vu-\vu_1\in W^m_p(\Omega)$ satisfies the equation with the zero right-hand side. Following the bootstrap argument in Case 2, we infer that $\vv\in W^m_2(\Omega)$. Therefore, by Case 1, $\vv$ must be a polynomial of degree at most $m-1$.

The corollary is proved.
\end{proof}

\begin{proof}[Proof of Theorem \ref{thmB}]
The theorem is proved in the same way as Corollary  \ref{cor7} in Cases 2 and 3 by using the classical $W^1_2$-estimate of the conormal derivative problem on a domain with a finite measure; see Theorem 13 of \cite{DongKim08a}. We remark that although in Theorem 13 (i) of \cite{DongKim08a} it is assumed that $b_i=c=0$, the same proof works under the relaxed condition $-D_ib_i+c=0$ in $\Omega$ and $b_in_i=0$ on $\partial\Omega$ in the weak sense, i.e., Assumption ($\text{H}^*$).
\end{proof}


\section*{Acknowledgement}

The authors are sincerely grateful to the referee for his careful reading and many helpful comments and suggestions.

\bibliographystyle{amsplain}

\begin{thebibliography}{1}

\bibitem{A65} \textsc{Agmon S.}:
\textit{Lectures on elliptic boundary value problems.}
Prepared for publication by B. Frank Jones, Jr. with the assistance of George W. Batten, Jr. Revised edition of the 1965 original. AMS Chelsea Publishing, Providence, RI, 2010.

\bibitem{ADN64} \textsc{Agmon S., Douglis A., Nirenberg L.}: Estimates near the boundary for solutions of elliptic partial differential equations satisfying general boundary
conditions, I, \textit{Comm. Pure Appl. Math.} \textbf{12}, 623--727  (1959); II, ibid., \textbf{17}, 35--92  (1964).

\bibitem{BW05} \textsc{Byun S., Wang L.}:
The conormal derivative problem for elliptic equations with BMO coefficients on Reifenberg flat domains, \textit{Proc. London Math. Soc. (3)} \textbf{90} (2005), no. 1, 245--272.

\bibitem{BW07}  \textsc{Byun S., Wang L.}: Parabolic equations in time dependent Reifenberg domains, \textit{Adv. Math.} \textbf{212} (2007), no. 2, 797--818.

\bibitem{BW10}  \textsc{Byun S., Wang L.}: Elliptic equations with measurable coefficients in Reifenberg domains, \textit{Adv. Math.} \textbf{225} (2010), no. 5, 2648--2673.

\bibitem{CaPe98} \textsc{Caffarelli L. A., Peral I.}: On $W^{1,p}$ estimates for elliptic equations in divergence form, \textit{Comm. Pure Appl. Math.} \textbf{51} (1998), no. 1, 1--21.

\bibitem{CFL1} \textsc{Chiarenza F., Frasca M., Longo P.}: Interior $W^{2,p}$ estimates for nondivergence elliptic
equations with discontinuous coefficients,
\textit{Ricerche Mat.} \textbf{40} (1991), 149--168.

\bibitem{CFL2} \textsc{Chiarenza F., Frasca M., Longo P.}: $W^{2,p}$-solvability of the Dirichlet problem for
nondivergence elliptic equations with VMO coefficients,
\textit{Trans. Amer. Math. Soc.} \textbf{336}, no. 2,  841--853  (1993).

\bibitem{CFF} \textsc{Chiarenza F., Franciosi M., Frasca  M.}:
$L\sp p$-estimates for linear elliptic systems with discontinuous coefficients, \textit{Atti Accad. Naz. Lincei Cl. Sci. Fis. Mat. Natur. Rend. Lincei (9) Mat. Appl.} \textbf{5} (1994), no. 1, 27--32. 

\bibitem{CKV} \textsc{Chipot M., Kinderlehrer D., Vergara-Caffarelli G.}:, Smoothness of linear laminates, \textit{Arch. Ration. Mech. Anal.} \textbf{96} (1986), no. 1, 81--96.

\bibitem{DM93} \textsc{DiBenedetto E., Manfredi J.}: On the higher integrability of the gradient of weak solutions of certain degenerate elliptic systems, \textit{Amer. J. Math.} \textbf{115} (1993), no. 5, 1107--1134.

\bibitem{DongKim08a} \textsc{Dong H., Kim D.}: Elliptic equations in divergence form with partially BMO coefficients, \textit{Arch. Ration. Mech. Anal.} \textbf{196} (2010), no. 1, 25--70.

\bibitem{DK09_02} \textsc{Dong H., Kim D.}: Parabolic and elliptic systems in divergence form with variably partially BMO coefficients, \textit{SIAM J. Math. Anal.} \textbf{43} (2011), no. 3, 1075--1098.

\bibitem{DK09_01} \textsc{Dong H., Kim D.}: On the $L_p$-solvability of higher order parabolic and elliptic systems with BMO coefficients, \textit{Arch. Rational Mech. Anal.} \textbf{199} (2011), no. 3, 889--941.

\bibitem{DK10} \textsc{Dong H., Kim D.}: Higher order elliptic and parabolic systems with variably partially BMO coefficients in regular and irregular domains, \textit{J. Funct. Anal.} \textbf{261}, (2011) no. 11, 3279--3327.


\bibitem{Fried} \textsc{Friedman A.}: \textit{Partial Differential Equations of Parabolic Type.}, Prentice-Hall, Englewood Cliffs, N.J, 2008. 

\bibitem{HHH} \textsc{Haller-Dintelmann R., Heck  H., Hieber  M.}:
$L^p$--$L^q$-estimates for parabolic systems in non-divergence form with VMO coefficients, \textit{J. London Math. Soc.}, (2) 74 (3), 717--736  (2006).

\bibitem{Iw83} \textsc{Iwaniec T.} Projections onto gradient fields and $L^p$-estimates for degenerated elliptic operators, \textit{Studia Math.} \textbf{75} (1983), no. 3, 293--312.

\bibitem{KS80} \textsc{Krylov N.V., Safonov M.V.}:
A certain property of solutions of parabolic equations
with measurable coefficients,
\textit{Izvestiya Akademii Nauk SSSR, seriya matematicheskaya}
\textbf{44} (1980),  no. 1,  161--175 in Russian; English
translation in \textit{Math. USSR Izvestija}
\textbf{16} (1981), no. 1, 151--164.

\bibitem{Krylov_2005} \textsc{Krylov N.~V.}: Parabolic and elliptic
equations with VMO coefficients, \textit{Comm. Partial Differential
Equations} \textbf{32}, no. 1-3, 453--475  (2007).

\bibitem{Krylov08}  \textsc{Krylov N.~V.}: Second-order elliptic equations with variably partially VMO coefficients, \textit{J. Funct. Anal.} \textbf{257}, 1695--1712  (2009).

\bibitem{LSU} \textsc{Lady\v{z}enskaja O. A., Solonnikov V. A., Ural'ceva N. N.}: \textit{Linear and quasilinear equations of parabolic type}.
American Mathematical Society, Providence, RI, 1967.

\bibitem{Lieb1} \textsc{Lieberman G. M.}: The conormal derivative problem for elliptic equations of variational type,
\textit{J. Differential Equations} \textbf{49}, no. 2,  218--257  (1983).

\bibitem{Lie85} \textsc{Lieberman G. M.}: Regularized distance and its applications, \textit{Pacific J. Math.} \textbf{117}
(1985), no. 2, 329--352.

\bibitem{Lieb2} \textsc{Lieberman G. M.}: The conormal derivative problem for equations of variational type in nonsmooth domains,
\textit{Trans. Amer. Math. Soc.} \textbf{330}, no. 1, 41--67  (1992).

\bibitem{Mi06} \textsc{Miyazaki Y.}: Higher order elliptic operators of divergence form in $C^1$ or Lipschitz domains, \textit{J. Differential Equations}  \textbf{230},  no. 1, 174--195  (2006).

\bibitem{PS3} \textsc{Palagachev D., Softova L.}: Precise regularity of solutions to elliptic systems with discontinuous data,  \textit{Ricerche Mat.}  \textbf{54}  (2005),  no. 2, 631--639 (2006)

\bibitem{Sa80} M.V. Safonov,
Harnack inequality for elliptic equations and the
H\"older property of their solutions, \textit{Zap. Nauchn. Sem. Leningrad.
Otdel. Mat. Inst. Steklov. (LOMI)} \textbf{96} (1980), 272--287 in Russian;
English translation in  \textit{J. Soviet Math.} \textbf{21} (1983), no. 5, 851--863.

\bibitem{Solo} \textsc{Solonnikov V. A.}:
On boundary value problems for linear parabolic systems of differential equations of general form, (Russian), \textit{Trudy Mat. Inst. Steklov.} \textbf{83}, 3--163  (1965);
English translation:
\textit{Proceedings of the Steklov Institute of Mathematics. No. 83 (1965): Boundary value problems of mathematical physics. III}.
Edited by O. A. Lady\v zenskaja. Translated from the Russian by A. Jablonski\u\i, American Mathematical Society, Providence, R.I. 1967 iv+184 pp. 

\end{thebibliography}

\def\cprime{$'$}\def\cprime{$'$} \def\cprime{$'$} \def\cprime{$'$}
  \def\cprime{$'$} \def\cprime{$'$}



\end{document}